\documentclass[11pt]{amsart}
\usepackage{hyperref,enumitem,graphicx}
\usepackage{amssymb,amsmath,amsrefs}
\usepackage{lipsum}
\usepackage[htt]{hyphenat}

\newcommand{\R}{\mathbb{R}}
\newcommand{\C}{\mathbb{C}}
\newcommand{\Q}{\mathbb{Q}}
\newcommand{\Z}{\mathbb{Z}}
\newcommand{\F}{\mathbb{F}}
\newcommand{\p}{\mathfrak p}
\newcommand{\PP}{\mathbb P}
\renewcommand{\q}{\mathfrak q}
\renewcommand{\P}{\mathfrak P}
\newcommand{\calG}{\mathcal G}
\renewcommand{\O}{\mathcal O}
\newcommand{\calC}{\mathcal C}
\newcommand{\calX}{\mathcal X}
\newcommand{\WW}{\mathcal W}
\newcommand{\ZZ}{\mathcal Z}

\newcommand{\into}{\hookrightarrow}
\newcommand{\onto}{\twoheadrightarrow}
\renewcommand{\to}{\rightarrow}
\renewcommand{\epsilon}{\varepsilon}

\newcommand{\Aut}{\operatorname{Aut}}
\newcommand{\Gal}{\operatorname{Gal}}
\newcommand{\Sym}{\operatorname{Sym}}

\newtheorem{thm}{Theorem}[section]
\newtheorem{conj}[thm]{Conjecture}
\newtheorem{lem}[thm]{Lemma}
\newtheorem{prop}[thm]{Proposition}
\newtheorem{cor}[thm]{Corollary}
\theoremstyle{definition}
\newtheorem{rem}[thm]{Remark}
\newtheorem{prob}[thm]{Problem}
\numberwithin{equation}{section}

\title[Specializations of dynatomic polynomials]{A finiteness theorem for specializations of dynatomic polynomials}
\author{David Krumm}
\address{Mathematics Department\\
Reed College}
\email{dkrumm@reed.edu}
\urladdr{http://maths.dk}

\begin{document}
\maketitle

\begin{abstract}
Let $t$ and $x$ be indeterminates, let $\phi(x)=x^2+t\in\Q(t)[x]$, and for every positive integer $n$ let $\Phi_n(t,x)$ denote the $n^{\text{th}}$ dynatomic polynomial of $\phi$. Let $G_n$ be the Galois group of $\Phi_n$ over the function field $\Q(t)$, and for $c\in\Q$ let $G_{n,c}$ be the Galois group of the specialized polynomial $\Phi_n(c,x)$. It follows from Hilbert's irreducibility theorem that for fixed $n$ we have $G_n\cong G_{n,c}$ for every $c$ outside a thin set $E_n\subset\Q$. By earlier work of Morton (for $n=3$) and the present author (for $n=4$), it is known that $E_n$ is infinite if $n\le 4$. In contrast, we show here that $E_n$ is finite if $n\in\{5,6,7,9\}$. As an application of this result we show that, for these values of $n$, the following holds with at most finitely many exceptions: for every $c\in\Q$, more than $81\%$ of prime numbers $p$ have the property that the polynomial $x^2+c$ does not have a point of period $n$ in the $p$-adic field $\Q_p$.
\end{abstract}

\section{Introduction}\label{intro_section}

Let $c$ be a rational number and let $\phi_c(x)=x^2+c$. Given any algebraic number $x_0$, we may consider the sequence $x_0, \phi_c(x_0),\phi_c(\phi_c(x_0)),\ldots$. If this sequence is periodic with period $n$, we say that $x_0$ has period $n$ under iteration of $\phi_c$. By allowing $c$ and $x_0$ to vary in $\Q$, one can find examples where $x_0$ has period 1, 2, or 3 under $\phi_c$. For instance, the pairs
\[(c,x_0)=(0,0), \;(-1,0),\; (-29/16, 5/4)\]
provide examples of periods 1, 2, and 3, respectively.

In \cite{poonen_prep} Poonen conjectured that if $n>3$, then there does not exist $c\in\Q$ such that the polynomial $\phi_c$ has a rational point of period $n$. This has been proved for periods 4 and 5, and also for period 6 assuming the BSD conjecture; see \cite{morton_period4, flynn-poonen-schaefer, stoll_6cycles}. The present paper is concerned with a strong form of Poonen's conjecture which was stated by the author in \cite{krumm_lgp}: if $n>3$, then for every $c\in\Q$ there exist infinitely many primes $p$ such that $\phi_c$ does not have a point of period $n$ in the $p$-adic field $\Q_p$. In fact, we will consider here a further strengthening of Poonen's conjecture.

\begin{conj}\label{positivity_conjecture}
Fix $n>3$. For every $c\in\Q$, let $T_{n,c}$ denote the set of primes $p$ such that $\phi_c$ does not have a point of period $n$ in $\Q_p$, and let $\delta(T_{n,c})$ be the Dirichlet density of $T_{n,c}$. Then $\delta(T_{n,c})>0$ for all $c\in\Q$.
\end{conj}

In order to study these conjectures it is useful to consider a family of \emph{dynatomic polynomials} defined as follows. For every positive integer $n$ we define a two-variable polynomial $\Phi_n\in\Q[t,x]$ by the formula
\begin{equation}\label{dynatomic_definition}
\Phi_n(t,x)=\prod_{d|n}\left(\phi^d(x)-x\right)^{\mu(n/d)},
\end{equation}
where $\mu$ is the M\"{o}bius function, $\phi(x)=x^2+t\in\Q(t)[x]$, and $\phi^d$ denotes the $d$-fold composition of $\phi$ with itself. The key property linking $\Phi_n$ to the above conjectures is that, for fixed  $c\in\Q$, every algebraic number having period $n$ under iteration of $\phi_c$ is a root of $\Phi_n(c,x)$, and conversely, every root of $\Phi_n(c,x)$ has period $n$ under $\phi_c$ except in rare cases when the period may be smaller than $n$; see \cite[Thm. 2.4]{morton-patel} for further details. 

Questions about the points of period $n$ under $\phi_c$ can thus be phrased as questions about the roots of $\Phi_n(c,x)$. It is therefore to be expected that a good understanding of the Galois group of $\Phi_n(c,x)$ will yield substantial information about the dynamical properties of the map $\phi_c$. The results of the article \cite{krumm_fourth_dynatomic} provide an example of the type of information that can be obtained in this way. By a careful analysis of how the Galois group of $\Phi_4(c,x)$ can change as $c$ varies in $\Q$, it is proved there that if $\alpha\in\bar\Q$ has period four under a map $\phi_c$, then the degree $[\Q(\alpha):\Q]$ can only be $2,4,8$, or 12; in particular the degree cannot be 1, which implies that $\phi_c$ does not have a rational point of period 4. Furthermore, the Galois group data is used to show that $\delta(T_{4,c})>0.39$ for every $c\in\Q$, thus proving Conjecture \ref{positivity_conjecture} for $n=4$. Motivated by these results, we are led to the following problem.

\begin{prob}\label{galois_family_problem}
Let $G_{n,c}$ denote the Galois group of $\Phi_n(c,x)$ over $\Q$. For fixed $n$, determine the structure of all the groups $G_{n,c}$ as $c$ varies in $\Q$.
\end{prob}

Since the polynomials $\Phi_n(c,x)$ for $c\in\Q$ are specializations of $\Phi_n$, it follows from Hilbert's irreducibility theorem \cite[Prop. 3.3.5]{serre_topics} that for every rational number $c$ outside a thin subset of $\Q$, the group $G_{n,c}$ is isomorphic to the Galois group of $\Phi_n$ over the function field $\Q(t)$. Moreover, by work of Bousch~\cite[Chap. 3]{bousch} it is known that $\Phi_n$ is irreducible and that its Galois group, which we denote by $G_n$, is isomorphic to a wreath product of a cyclic group and a symmetric group; indeed,  $G_n\cong(\Z/n\Z)\wr S_r$, where $rn=\deg\Phi_n$. Hence, for most $c\in\Q$ the structure of $G_{n,c}$ is known. However, a complete solution of Problem \ref{galois_family_problem} would require understanding precisely for which numbers $c$ the specialization $t\mapsto c$ fails to preserve the Galois group of $\Phi_n$. This raises a new but closely related problem.

\begin{prob}\label{galois_specialization_problem}
For fixed $n$, determine all $c\in\Q$ such that $G_{n,c}\not\cong G_n$.
\end{prob}

Let $E_n=\{c\in\Q\;\vert\;G_{n,c}\not\cong G_n\}$. By work of Morton~\cite{morton_period3} and the present author~\cite{krumm_fourth_dynatomic}, the sets $E_n$ are well understood for $n\le 4$; in particular, one notable feature of these sets is that they are infinite. In contrast, empirical evidence suggests that $E_n$ is finite for every $n>4$. The main purpose of this article is to prove this finiteness statement for several values of $n$.

\begin{thm}\label{finiteness_theorem_intro}
The set $E_n$ is finite if $n\in\{5,6,7,9\}$.
\end{thm}

Using this theorem we can provide further evidence in support of Conjecture \ref{positivity_conjecture}. It follows from the theorem that, for the above values of $n$, we have $G_{n,c}\cong(\Z/n\Z)\wr S_r$ for all but finitely many $c\in\Q$. Excluding this finite set we therefore know the structure of all the Galois groups $G_{n,c}$. The Chebotarev density theorem can then be used to determine the value of $\delta(T_{n,c})$ by a straightforward calculation within the group $(\Z/n\Z)\wr S_r$. In this way we obtain the following result.

\begin{thm}\label{density_thms_intro}
There exists a finite set $E\subset\Q$ such that the following lower bounds hold for every $c\in\Q\setminus E$:
\[\delta(T_{5,c})>0.81,\;\delta(T_{6,c})>0.84,\;\delta(T_{7,c})>0.86,\;\delta(T_{9,c})>0.89.\]
\end{thm}

The proof of Theorem \ref{finiteness_theorem_intro} relies on Hilbert's irreducibility theorem and Faltings's theorem to reduce the proof to a problem of showing that certain algebraic curves have genera greater than 1. More precisely, let $S$ be a splitting field of $\Phi_n$ over $\Q(t)$, so that $G_n=\Gal(S/\Q(t))$, and let $\calX$ be the smooth projective curve over $\Q$ whose function field is $S$.  As explained in \S\ref{preliminaries_section}, in order to show that the set $E_n$ is finite it suffices to show that, for every maximal proper subgroup $M<G_n$, the quotient curve $\calX/M$ has genus greater than 1. Our main objective is therefore to compute the genera of these quotient curves, or at least to obtain lower bounds for them.

The methods we develop for this purpose allow us to reduce the problem to a series of computations within the groups $G_n$. For $n\in\{5,6\}$ we are able to determine the genera exactly, and for $n\in\{7,9\}$ we prove lower bounds which suffice for our purposes. Though the methods used here could in principle be used to extend our results to higher values of $n$, there are computational limitations which prevent this. For instance, the group $G_{11}$ has order $11^{186}(186)!$, and the cost of computing its maximal subgroups is prohibitively expensive. Other computational issues are discussed in \S\ref{bounds_section}.

Though it would be desirable to explicitly determine the finite sets $E_n$ in Theorem \ref{finiteness_theorem_intro}, our method of proof does not suggest a feasible way of doing this. Indeed, one would have to determine the sets of rational points on several curves of very large genera, a problem which seems impossible with current methods. Nevertheless, in \S\ref{exceptional_search_section} we make some elementary observations regarding the sets $E_n$; for instance, they are always nonempty.

This article is organized as follows. In \S\ref{preliminaries_section} we establish two foundational results for the rest of the article. In \S\ref{valuation_theory_section} we prove a theorem concerning the structure of inertia groups in Galois extensions of valued fields; this may be of independent interest. In \S\ref{dynatomic_section} we recall various properties of dynatomic polynomials which were mostly proved by P. Morton. In \S\ref{galois_action_section} we study the action of $G_n$ on the roots of $\Phi_n$. In \S\S\ref{computation_section}-\ref{bounds_section} we apply the results of earlier sections to carry out the genus computations from which Theorem \ref{finiteness_theorem_intro} can be deduced. In \S\ref{density_section} we prove Theorem \ref{density_thms_intro}. Finally, in \S\ref{exceptional_search_section} we list the known elements of the sets $E_n$.

\section{preliminaries}\label{preliminaries_section}

Let $n$ be a positive integer and let $\Phi_n$ be the polynomial  defined in \eqref{dynatomic_definition}. Let $S$ be a splitting field of $\Phi_n$ over $\Q(t)$, and $G_n=\Gal(S/\Q(t))$. Recall that $E_n$ denotes the set of all rational numbers $c$ such that $G_{n,c}\not\cong G_n$, where $G_{n,c}$ is the Galois group of $\Phi_n(c,x)$ over $\Q$. The following lemma provides sufficient conditions for $E_n$ to be a finite set. 

\begin{lem}\label{explicit_HIT_lem}
Let $M_1,\ldots, M_s$ be representatives of all the conjugacy classes of maximal subgroups of $G_n$, and let $L_i$ denote the fixed field of $M_i$. Suppose that every function field $L_i$ has genus greater than 1. Then $E_n$ is finite.
\end{lem}
\begin{proof}
Let $\calX$ be the smooth projective curve with function field $S$, and for every index $i$, let $\calX_i$ be the quotient curve $\calX/M_i$. It follows from the proof of Proposition 3.3.1 in \cite{serre_topics} (see also \cite[Thm. 1.1]{krumm-sutherland}) that there exist a finite set $\mathcal E\subset\PP^1(\Q)$ and morphisms $\pi_i:\calX_i\to\PP^1$ such that
\[E_n\subseteq \mathcal E\cup\bigcup_{i=1}^s\pi_i(\calX_i(\Q)).\]
Since $L_i$ is the function field of $\calX_i$, the hypotheses imply that the smooth projective model of $\calX_i$ has genus greater than 1, and hence, by Faltings's theorem \cite{faltings}, the set $\calX_i(\Q)$ is finite. The result follows immediately.
\end{proof}

In view of Lemma \ref{explicit_HIT_lem}, the main objects of interest in this article are the genera of the minimal intermediate fields in the extension $S/\Q(t)$. Our first step towards understanding these genera will be to show that in computing them we may replace $\Q$ with any subfield of $\C$.
 
\begin{prop}\label{base_change_prop}
Let $\F$ be any field satisfying $\Q\subseteq\F\subseteq\C$, and let $N$ be a splitting field of $\Phi_n$ over $\F(t)$. Then there is an isomorphism
\[\iota:\Gal(N/\F(t))\longrightarrow\Gal(S/\Q(t))\]
with the following property: if $A$ is a subgroup of $\Gal(N/\F(t))$ and $B=\iota(A)$, then the fixed fields of $A$ and $B$ have the same genus.
\end{prop}
\begin{proof}
Let $\Sigma$ be a splitting field of $\Phi_n$ over $\C(t)$, and let $R\subset\Sigma$ be the set of roots of $\Phi_n$. By basic field theory, we may identify $N$ with the field $\F(t)(R)$ and $S$ with the field $\Q(t)(R)$. Restriction of automorphisms then yields injective homomorphisms
\begin{equation}\label{galois_embeddings_auto}
\Gal(\Sigma/\C(t))\into\Gal(N/\F(t))\into\Gal(S/\Q(t)).
\end{equation}
The group $\Gal(\Sigma/\C(t))$ is naturally isomorphic to a subgroup $G_{\C}$ of the symmetric group $\Sym(R)$. (Explicitly, the isomorphism is given by restriction to $R$.) Similarly, we define groups $G_{\F}$ and $G_{\Q}$. By \eqref{galois_embeddings_auto} we have
\begin{equation}\label{galois_embeddings_perm}
G_{\C}\le G_{\F}\le G_{\Q}\le\Sym(R).
\end{equation}
The polynomial $\phi(x)=x^2+t$ permutes the elements of $R$ (see, for instance, \cite[\S2.2]{krumm_lgp}); thus we may regard $\phi$ as an element of the group $\Sym(R)$. Let $\calC$ denote the centralizer of $\phi$ in $\Sym(R)$. Since $\phi$ is a polynomial map, it commutes with every element of $\Gal(S/\Q(t))$, and therefore $G_{\Q}\le\calC$. Now, by Theorem 3 in \cite[Chap. 3]{bousch} we have $G_{\C}=\calC$. Hence, \eqref{galois_embeddings_perm} implies that $G_{\C}=G_{\F}=G_{\Q}$. It follows that the embeddings \eqref{galois_embeddings_auto} are in fact isomorphisms; in particular, restriction to $S$ is an isomorphism
\begin{equation}\label{iota_galois_map}
\iota:\Gal(N/\F(t))\stackrel{\sim}{\longrightarrow}\Gal(S/\Q(t)).
\end{equation}

We now digress briefly from the main proof.
\begin{lem}\label{constant_field_lem}
The field $\Q$ is algebraically closed in $S$.
\end{lem}
\begin{proof}
Let $k$ be the algebraic closure of $\Q$ in $S$. By general theory of algebraic function fields, the extension $k/\Q$ is finite; moreover, it can easily be shown to be a Galois extension. To see that $k/\Q$ is normal, let $p(x)\in\Q[x]$ be an irreducible polynomial having a root in $k$. Then $p$ remains irreducible in $\Q(t)[x]$ (see Lemma 3.1.10 in \cite{stichtenoth}) and has a root in $S$; therefore $p$ splits in $S$. However, by definition of $k$, every root of $p$ in $S$ belongs to $k$. Hence, $p$ splits in $k$.

Since $k(t)$ is the composite of $k$ and $\Q(t)$, the extension $k(t)/\Q(t)$ is Galois, and restriction to $k$ yields an isomorphism
\[\Gal(k(t)/\Q(t))\cong\Gal(k/k\cap\Q(t))=\Gal(k/\Q).\]
It follows that there is a surjective homomorphism $\Gal(S/\Q(t))\to\Gal(k/\Q)$ with kernel $H:=\Gal(S/k(t))$. Now, taking $\F=k$ in \eqref{iota_galois_map}, the image of $\iota$ is clearly contained in $H$, so that in fact $H=\Gal(S/\Q(t))$. Therefore $\Gal(k/\Q)$ must be trivial, and $k=\Q$.
\end{proof}

Returning to the proof of the proposition, let $A\le\Gal(N/\F(t))$ and set $B=\iota(A)$. Let $U$ and $V$ be the fixed fields of $A$ and $B$, respectively. Thus, $U$ and $V$ are intermediate fields in the extensions $N/\F(t)$ and $S/\Q(t)$. We claim that $U$ is the composite of $V$ and $\F$. The fact that $U\supseteq V$ follows immediately from the definitions, and it is clear that $U\supseteq\F$; hence $U\supseteq V\F$. To prove that $U=V\F$ we will show that $[U:\F(t)]=[V\F:\F(t)]$. Since $\iota$ is an isomorphism mapping $A$ to $B$, we have
\[[U:\F(t)]=|\Gal(N/\F(t)):A|=|\Gal(S/\Q(t)):B|=[V:\Q(t)].\]
Thus, it suffices to show that $[V:\Q(t)]=[V\F:\F(t)]$. Let $\alpha$ be a primitive element for $V$ over $\Q(t)$, and let $p \in\Q(t)[x]$ be the minimal polynomial of $\alpha$. Clearly $V\F=\F(t)(\alpha)$, so it is enough to show that $p$ remains irreducible over $\F(t)$. Since $p$ is irreducible over $\Q(t)$, the group $\Gal(S/\Q(t))$ acts transitively on the roots of $p$. This, together with the fact that $\iota$ is given by restriction to $N$, imply that $\Gal(N/\F(t))$ also acts transitively on the roots of $p$, and therefore $p$ is irreducible over $\F(t)$. This completes the proof that $U=V\F$.

It remains only to show that $U$ and $V$ have the same genus. Since $\F$ contains the constant field of $V$ (by Lemma \ref{constant_field_lem}), $U=V\F$ is a constant field extension of $V$ (in the terminology of \cite[\S3.6]{stichtenoth}). Equality between the genera of $U$ and $V$ now follows from
Theorem 22 in \cite[p. 291]{artin}; see also Theorem 3.6.3 in \cite{stichtenoth}.
\end{proof}

From Lemma \ref{explicit_HIT_lem} and Proposition \ref{base_change_prop} we deduce the following proposition, which is the key result of this section.

\begin{prop}\label{En_finiteness_prop}
Let $N$ be a splitting field of $\Phi_n$ over $\bar\Q(t)$. Let $M_1,\ldots, M_s$ be representatives of all the conjugacy classes of maximal subgroups of the group $G=\Gal(N/\bar\Q(t))$, and let $L_i$ be the fixed field of $M_i$. Suppose that the genus of $L_i$ is greater than 1 for every index $i$. Then the set $E_n$ is finite.
\end{prop}

\section{A result in valuation theory}\label{valuation_theory_section}

Let $K$ be a field, and let $v:K^{\ast}\to\R$ be a discrete valuation of $K$ with perfect residue field $k$. Let $N$ be a finite Galois extension of $K$ with Galois group $G=\Gal(N/K)$. For any elements $\sigma,\tau\in G$ we will write $\tau^{\sigma}$ to denote the conjugate $\sigma^{-1}\tau\sigma$; similarly, for any subgroup $A\le G$ we let $A^{\sigma}=\sigma^{-1}A\sigma$.

If $L$ is an intermediate field in the extension $N/K$ and $w$ is a valuation of $N$ extending a valuation $u$ of $L$, we denote by $D_{w|u}$ and $I_{w|u}$ the decomposition and inertia groups of $w$ over $u$. If $u$ extends the valuation $v$ of $K$, we let $e_{u|v}$ and $f_{u|v}$ denote the ramification index and residue degree of $u$ over $v$.

\begin{lem}\label{double_coset_bijection_lem}
Let $w$ be a valuation of $N$ extending $v$, and let $D=D_{w|v}$ and $I=I_{w|v}$. Let $H$ be a subgroup of $G$ with fixed field $L$, and let $S_L$ be the set of all valuations of $L$ extending $v$. Then there is a well-defined bijection
\[D\backslash G/H\stackrel{\sim}{\longrightarrow} S_L\]
given by $D\sigma H\mapsto (w\circ\sigma)|_L$. Furthermore, if $u=(w\circ\sigma)|_L$, then
\begin{equation}\label{double_coset_ef}
e_{u|v}\cdot f_{u|v}=|D^{\sigma}:D^{\sigma}\cap H|\;\;\text{and}\;\; e_{u|v}=|I^{\sigma}:I^{\sigma}\cap H|.
\end{equation}
\end{lem}
\begin{proof}
The first statement is well known; a proof may be found in Lemma 17.1.2 and Corollary 17.1.3 of \cite{efrat}. Suppose now that $u=(w\circ\sigma)|_L$, and let $\tilde w=w\circ\sigma$. It is then a simple exercise to show that 
\begin{equation}\label{decomposition_inertia_formula}
D_{\tilde w|u}=D^{\sigma}\cap H\;\;\text{and}\;\;I_{\tilde w|u}=I^{\sigma}\cap H.
\end{equation}
Note that $D^{\sigma}=D_{\tilde w|v}$ and $I^{\sigma}=I_{\tilde w|v}$. Now, since $k$ is perfect, we have $|D_{\tilde w|v}|=e_{\tilde w|v}\cdot f_{\tilde w|v}$ and $|I_{\tilde w|v}|=e_{\tilde w|v}$ (see \cite[Chap. I, Prop. 9.6]{neukirch}). The relations \eqref{double_coset_ef} now follow easily from \eqref{decomposition_inertia_formula}.
\end{proof}

\begin{prop}\label{inertia_cycle_type_prop}
Suppose that $N$ is the splitting field of an irreducible polynomial $P(x)\in K[x]$. Let $F$ be a subextension of $N/K$ obtained by adjoining one root of $P(x)$ to $K$. Let $u_1,\ldots, u_m$ be the distinct valuations of $F$ extending $v$, and set $e_i=e_{u_i|v}$ and $f_i=f_{u_i|v}$. Let $w$ be a valuation of $N$ extending $v$, and assume that $e_{w|v}$ is not divisible by the characteristic of $k$. Then the inertia group $I_{w|v}$ is generated by an element whose disjoint cycle decomposition (as a permutation of the roots of $P$) has the form
\begin{equation}\label{inertia_generator_cycle_type}
\underbrace{(e_1\text{-cycle})\cdots(e_1\text{-cycle})}_{f_1\;\text{times}}\cdots\underbrace{(e_m\text{-cycle})\cdots(e_m\text{-cycle})}_{f_m\;\text{times}}.
\end{equation}
\end{prop}
\begin{proof}
Set $D=D_{w|v}$ and $I=I_{w|v}$. The assumption that the characteristic of $k$ does not divide $|I|$ implies that $I$ is a cyclic group; see \cite[Prop. 3.8.5]{stichtenoth} or \cite[\S16.2]{efrat}. Let $R$ denote the set of roots of $P(x)$ in $N$, and consider the natural action of $I$ on $R$. Let $\O$ be the set of orbits of this action. We will show that $\O$ can be partitioned into subsets $S_1,\ldots, S_m$ such that every orbit in $S_i$ has cardinality $e_i$, and $\#S_i=f_i$. Note that this implies that every generator of $I$ has a cycle decomposition of the form \eqref{inertia_generator_cycle_type}.

For every $x\in R$ let $\O_x$ and $I_x$, respectively, denote the orbit of $x$ (under the action of $I$) and the stabilizer of $x$ in $I$. Let $r\in R$ be such that $F=K(r)$, and set $H=\Gal(N/F)$. Note that $H$ is the stabilizer of $r$ in $G$.

By Lemma \ref{double_coset_bijection_lem}, there exist distinct double cosets $D\sigma_1H,\ldots, D\sigma_mH$ such that $u_i=(w\circ\sigma_i)|_F$. For $i=1,\ldots, m$ we define a map $\psi_i$ as follows:
\begin{align*}
I^{\sigma_i}\backslash D^{\sigma_i}/(D^{\sigma_i}\cap H)\;&\stackrel{\psi_i}{\longrightarrow}\;\O, \\
I^{\sigma_i}\tau (D^{\sigma_i}\cap H) &\longmapsto \O_{\sigma_i\tau(r)}.
\end{align*}
A straightforward calculation shows that $\psi_i$ is well defined and injective. Letting $S_i\subseteq\O$ be the image of $\psi_i$, we claim that the sets $S_1,\ldots, S_m$ have the properties stated above.

We begin by showing that every orbit in $S_i$ has cardinality $e_i$. To ease notation, let us fix an index $i$ and set $\sigma=\sigma_i$ and $M=D^{\sigma}\cap H$. Letting $\tau\in D^{\sigma}$, we must show that $\#\O_{\sigma\tau(r)}=e_i$. Note that $(I_{\sigma\tau(r)})^{\sigma\tau}=I^{\sigma\tau}\cap H$, so that $|I_{\sigma\tau(r)}|=|I^{\sigma\tau}\cap H|$, and therefore
\[\#\O_{\sigma\tau(r)}=|I:I_{\sigma\tau(r)}|=\frac{|I|}{|I_{\sigma\tau(r)}|}=\frac{|I^{\sigma\tau}|}{|I_{\sigma\tau(r)}|}=\frac{|I^{\sigma\tau}|}{|I_{\sigma\tau}\cap H|}=|I^{\sigma\tau}:I^{\sigma\tau}\cap H|.\]
Now, since $\tau\in D^{\sigma}$, we have $\sigma\tau\in D\sigma H$. Lemma \ref{double_coset_bijection_lem} then implies that $(w\circ\sigma\tau)|_F=(w\circ\sigma)|_F=u_i$ and $|I^{\sigma\tau}:I^{\sigma\tau}\cap H|=e_i$. Hence $\#\O_{\sigma\tau(r)}=e_i$.

Next we show that $\#S_i=f_i$. Note that $\#S_i=\#I^{\sigma}\backslash D^{\sigma}/M$ since $\psi_i$ is injective. The fact that $I$ is a normal subgroup of $D$ implies that
\[I^{\sigma}\backslash D^{\sigma}/M=D^{\sigma}/(I^{\sigma}M).\]
Thus, using Lemma \ref{double_coset_bijection_lem} we obtain
\[\#S_i=|D^{\sigma}|/|I^{\sigma}M|=\frac{|D^{\sigma}|\cdot |I^{\sigma}\cap H|}{|D^{\sigma}\cap H|\cdot|I^{\sigma}|}=\frac{|D^{\sigma}:D^{\sigma}\cap H|}{|I^{\sigma}:I^{\sigma}\cap H|}=\frac{e_if_i}{e_i}=f_i.\]

Now we show that the sets $S_1,\ldots, S_m$ are pairwise disjoint. Suppose, by contradiction, that there exist distinct indices $i, j$ such that $S_i\cap S_j\ne\emptyset$. Then there exist $\alpha\in D^{\sigma_i}$, $\beta\in D^{\sigma_j}$, and $\gamma\in I$ such that $\sigma_i\alpha(r)=\gamma\sigma_j\beta(r)$. Writing $\alpha=\sigma_i^{-1}\delta\sigma_i$ and $\beta=\sigma_j^{-1}d\sigma_j$ with $\delta,d\in D$, this implies that $\delta\sigma_i(r)=\gamma d\sigma_j(r)$; hence, there exists $h\in H$ such that $\sigma_i=\delta^{-1}\gamma d\sigma_jh$. Note that $\delta^{-1}\gamma d\in D$, so the previous equality implies that $\sigma_i\in D\sigma_jH$ and therefore $D\sigma_iH=D\sigma_jH$, a contradiction.

Finally, we show that $\O=\cup_{i=1}^mS_i$. Let $R_1,\ldots, R_m$ be the subsets of $R$ defined by $R_i=\cup_{C\in S_i}C$. From the results proved above it follows that $\#R_i=e_if_i$ and that the sets $R_1,\ldots, R_m$ are pairwise disjoint. Given that $v$ is a discrete valuation, we have the relation $[F:K]=\sum_{i=1}^me_if_i$. Hence
\[\#R=\deg(P)=[F:K]=\sum_{i=1}^me_if_i=\sum_{i=1}^m\#R_i=\#\bigcup_{i=1}^mR_i.\]
It follows that $R=\cup_{i=1}^mR_i$, which implies that $\O=\cup_{i=1}^mS_i$.
\end{proof}

\begin{rem}
Proposition \ref{inertia_cycle_type_prop} was inspired by a theorem of Beckmann~\cite{beckmann} concerning inertia groups in Galois extensions of $\Q$; indeed, Beckmann's result is essentially the case $K=\Q$ of the proposition. However, the proof given here has little in common with the proof in ~\cite{beckmann}.
\end{rem}

\begin{prop}\label{unramified_valuation_count}
With notation and assumptions as in Proposition \ref{inertia_cycle_type_prop}, let $\gamma$ be a generator of $I_{w|v}$ and let $H$ be a subgroup of $G$ with fixed field $L$. Suppose that $D_{w|v}=I_{w|v}$. Then the number of valuations $u$ of $L$ extending $v$ such that $e_{u|v}=1$ is given by
\begin{equation}\label{unramified_valuation_formula}
\frac{|C_G(\gamma)|\cdot s(H,\gamma)}{|H|},
\end{equation}
where $C_G(\gamma)$ is the centralizer of $\gamma$ in $G$ and $s(H,\gamma)$ is the number of $G$-conjugates of $\gamma$ that belong to $H$.
\end{prop}
\begin{proof}
Let $D=D_{w|v}$ and define sets $A=\{\sigma\in G\;\vert\;\gamma^{\sigma}\in H\}$ and
\[\Delta=\{D\sigma H\in D\backslash G/H\;\vert\;\sigma\in A\}.\]
It follows from Lemma \ref{double_coset_bijection_lem} that the cardinality of $\Delta$ is equal to the number of valuations $u$ of $L$ extending $v$ such that $e_{u|v}=1$. Thus, in order to prove the proposition it suffices to show that $|H|\cdot(\#\Delta)=|C_G(\gamma)|\cdot s(H,\gamma)$. 

For every element $a\in A$ the right coset $C_G(\gamma)\cdot a$ is contained in $A$; hence, the set $U=\{C_G(\gamma)\cdot a\;\vert\;a\in A\}$ is a partition of $A$ into subsets of size $|C_G(\gamma)|$. Thus $\#A=|C_G(\gamma)|\cdot(\#U)$. Now let $B=\{\gamma^{\sigma}\;\vert\;\sigma\in G\}\cap H$, so that $\#B=s(H,\gamma)$. Note that $\#U=\#B$; indeed, there is a bijective map $U\to B$ given by $C_G(\gamma)\cdot a\mapsto\gamma^a$. Therefore,
\begin{equation}\label{AB_size}
\#A=|C_G(\gamma)|\cdot(\#B)={|C_G(\gamma)|\cdot s(H,\gamma)}.
\end{equation}
Let $f:A\onto\Delta$ be the surjective map given by $f(\sigma)=D\sigma H$. We claim that, for every $a\in A$, $f^{-1}(f(a))=aH$. It is clear that $aH\subseteq f^{-1}(f(a))$. Now suppose that $f(a')=f(a)$, so that $a'=dah$ for some $d\in D$ and $h\in H$. Since $\gamma^a\in H$, we may write $\gamma a=ah'$ for some $h'\in H$. Furthermore, since $D=I_{w|v}=\langle\gamma\rangle$, we have $d=\gamma^n$ for some positive integer $n$. Thus \[a'=dah=\gamma^nah=a(h')^nh\in aH,\]
which proves the claim. Since every fiber of $f$ has cardinality $|H|$, we have $\#A=|H|\cdot(\#\Delta)$, and hence, by \eqref{AB_size}, $|H|\cdot (\#\Delta)=|C_G(\gamma)|\cdot s(H,\gamma)$.
\end{proof}

For later reference, we include here a combined statement of Propositions \ref{inertia_cycle_type_prop} and \ref{unramified_valuation_count} in the special case where $K$ is the function field $\bar\Q(t)$ and the valuation $v$ corresponds to a place $p$ of $K$. Note that in this case all residue degrees $f_{u|v}$ are equal to 1.

\begin{cor}\label{function_field_inertia_cor}
Let $t$ be an indeterminate and $K=\bar\Q(t)$. Suppose that $P(x)\in K[x]$ is irreducible, and let $N$ be a splitting field for $P(x)$. Let $F$ be a subextension of $N/K$ obtained by adjoining one root of $P(x)$ to $K$. Let $p$ be a place of $K$, and let $\p_1,\ldots, \p_m$ be the distinct places of $F$ lying over $p$. Then, for every place $\P$ of $N$ lying over $p$, the inertia group $I_{\P|p}$ is generated by an element $\gamma$ whose disjoint cycle decomposition has the form $(e_1\text{-cycle})\cdots(e_m\text{-cycle})$, where $e_i$ is the ramification index of $\p_i$ over $p$. Furthermore, if $H$ is a subgroup of $G=\Gal(N/K)$ with fixed field $L$, then the number of places of $L$ lying over $p$ which are unramified over $K$ is given by the formula \eqref{unramified_valuation_formula}. 
\end{cor}

\section{Ramification data for dynatomic polynomials}\label{dynatomic_section}
Let us fix a positive integer $n$. We will henceforth regard the polynomial $\Phi_n(x)$ as an element of the ring $\bar\Q(t)[x]$. As such, it is known by work of Bousch~\cite[Chap. 3]{bousch} that $\Phi_n$ is irreducible. In this section we will apply Corollary \ref{function_field_inertia_cor} to study inertia groups in the Galois group of $\Phi_n$.

Let $K=\bar\Q(t)$, let $N/K$ be a splitting field of $\Phi_n$, and let $G=\Gal(N/K)$. Let $F$ be a subextension of $N/K$ obtained by adjoining one root of $\Phi_n$ to $K$. In \cite[\S3]{morton_dynatomic_curves} Morton studies the ramification of places in the extension $F/K$ by using certain polynomials $\Delta_{n,d}\in\Z[t]$, where $d$ is a divisor of $n$. These polynomials had previously been defined in \cite[\S1]{morton-vivaldi}; we refer the reader to that article for the definition. We now recall a few results from \cite{morton_dynatomic_curves} and \cite{morton-vivaldi} which will be needed here.

For every positive integer $s$, let
\[\nu(s)=\frac{1}{2}\sum_{d|s}\mu(s/d)2^d.\]

\begin{lem}[Morton-Vivaldi]\label{delta_roots_disjoint_lem}
For every divisor $d$ of $n$, let $R_{n,d}\subset\bar\Q$ denote the set of roots of $\Delta_{n,d}$. Then the following hold:
\begin{enumerate}
\item[(a)] $\#R_{n,d}=\deg\Delta_{n,d}$ for every $d$.
\item[(b)] If $d$ and $e$ are distinct divisors of $n$, then $R_{n,d}\cap R_{n,e}=\emptyset$.
\item[(c)] Letting $\varphi$ denote Euler's phi function, the degree of $\Delta_{n,d}$ is given by
\[\deg\Delta_{n,d}=
\begin{cases}
\nu(d)\varphi(n/d) &\text{if } d<n\\
\nu(n)-\sum_{\substack{k|n \\ k<n}}\nu(k)\varphi(n/k) & \text{if } d=n.
\end{cases}\]
\end{enumerate}
\end{lem}
\begin{proof}
All statements are proved in \cite{morton-vivaldi}. Indeed, (a) and (b) follow from Proposition 3.2, and (c) follows from Corollary 3.3.
\end{proof}

Recall that for every place $p$ of $K$, the \emph{conorm} of $p$ with respect to the extension $F/K$ is the divisor, which we write multiplicatively, defined by
\[i_{F/K}(p)=\p_1^{e_1}\cdots\p_s^{e_s},\] where $\p_1,\ldots,\p_s$ are the distinct places of $F$ lying over $p$ and $e_i$ is the ramification index of $\p_i$ over $p$. A discussion of the basic properties of the conorm map may be found in \cite[\S3.1]{stichtenoth} or \cite[Chap. 7]{rosen}.

Let $D=\deg\Phi_n$; note that $D=2\nu(n)$. As explained in \S\ref{galois_action_section}, the set of roots of $\Phi_n$ can be partitioned into sets of cardinality $n$, and therefore $n$ divides $D$. Let $r=D/n$.

\begin{lem}[Morton]\label{morton_ramification_lem}
Let $p_{\infty}$ be the infinite place of $K$, i.e., the place corresponding to the valuation $v_{\infty}$ of $K$ given by $v_{\infty}(f/g)=\deg g-\deg f$. For $b\in\bar\Q$, let $p_b$ denote the place of $K$ corresponding to the polynomial $t-b$.
\begin{enumerate}
\item [(a)] The places of $K$ that ramify in $F$ are $p_{\infty}$ and $p_b$ for $b\in\cup_{d|n}R_{n,d}$.
\item[(b)] The conorm of $p_{\infty}$ has the form
\[i_{F/K}(p_{\infty})=\p_1^2\cdots\p_{\nu(n)}^2.\]
\item[(c)] For every $b\in R_{n,n}$, the conorm of $p_b$ has the form
\[i_{F/K}(p_b)=\p_1^2\cdots\p_n^2\cdot\q_1\cdots\q_{n(r-2)}.\]
\item[(d)] For every $b\in R_{n,d}$, where $d<n$, the conorm of $p_b$ has the form
\[i_{F/K}(p_b)=\p_1^{n/d}\cdots\p_d^{n/d}\cdot\q_1\cdots\q_{n(r-1)}.\]
\end{enumerate}
\end{lem}
\begin{proof}
All statements are proved in \cite{morton_dynatomic_curves}; (a), (c) and (d) follow from the proof of Proposition 9, and (b) follows from Proposition 10.
\end{proof}

Let $\PP=\{p_{\infty}\}\cup\{p_b\;\vert\;b\in \cup_{d|n}R_{n,d}\}$ be the set of places of $K$ that ramify in $F$. For any intermediate field $L$ in the extension $N/K$ and any place $p$ of $K$, let $\PP_L(p)$ denote the set of places of $L$ lying over $p$.

We introduce some terminology to be used throughout the article. Suppose that $G$ is a group acting on a finite set $X$, and let $g\in G$. We say that $g$ has \emph{cycle type} $(a,b)$, where $a$ and $b$ are positive integers, if the disjoint cycle decomposition of $g$, disregarding 1-cycles, is a product of $b$ $a$-cycles.

Applying Corollary \ref{function_field_inertia_cor} to the polynomial $\Phi_n$ and using Lemma \ref{morton_ramification_lem}, we immediately obtain the following description of inertia groups in $G$.

\begin{prop}\label{dynatomic_inertia_prop} Let $p\in\PP$ and $\P\in\PP_N(p)$. Then the inertia group $I_{\P|p}$ has a generator with cycle type $(a,b)$ satisfying
\[(a,b)=
\begin{cases}
(2,D/2) & \text{ if } p=p_{\infty};\\
(2,n) & \text{ if } p=p_b \text{ with } b\in R_{n,n};\\
(n/d,d) & \text{ if } p=p_b \text{ with } b\in R_{n,d}, d<n.
\end{cases}\]
\end{prop}

In addition to the data on ramification of places in $F/K$ provided by Lemma \ref{morton_ramification_lem}, in later sections we will need some ramification data for a subfield $F_0\subset F$ defined as follows. Let $\theta$ be a root of $\Phi_n$ such that $F=K(\theta)$. The field $F$ has an automorphism\footnote{Note that $\phi(\theta)$ is a root of $\Phi_n$, so there is an isomorphism $F\to K(\phi(\theta))$ mapping $\theta$ to $\phi(\theta)$. Moreover, the fact that $\phi^n(\theta)=\theta$ implies that $F=K(\phi(\theta))$, so this map is in fact an automorphism of $F$.} given by $\theta\mapsto\phi(\theta)=\theta^2+t$; we define $F_0$ to be the fixed field of this automorphism.

\begin{prop}[Morton]\label{F0_ramification_prop}Let $p\in\PP$ and let $S(p)=\sum_{\q\in\PP_{F_0}(p)}(e_{\q|p}-1)$.
\begin{enumerate}
\item[(a)] If $p=p_{\infty}$, then $S(p)=r-e_n$, where
\[e_n=\frac{1}{2n}\sum_{d|(n,2)}\varphi(d)^2\cdot\sum_{k\in U_{n,d}}\mu(n/k)2^{k/d}.\]
Here $U_{n,d}=\{k\in\Z_{>0}:k|n,\; d|k,\text{ and } (n/k,d)=1\}$.
\item[(b)] If $p=p_b$, where $b\in R_{n,n}$, then $S(p)=1$.
\item[(c)] If $p=p_b$, where $b\in R_{n,d}$ for some $d<n$, then $S(p)=0$.
\end{enumerate}
\end{prop}
\begin{proof}
All statements are proved in \cite{morton_dynatomic_curves}; (a) follows Theorem 13, while (b) and (c) can be deduced from the proof of Proposition 9. Indeed, it is shown in that proposition that if $p=p_b$, where $b\in R_{n,n}$, then there is a unique ramified place of $F_0$ lying over $p$, and its ramification index is 2; this implies (b). Similarly, If $p=p_b$, where $b\in R_{n,d}$ for some $d<n$, then $p$ is unramified in $F_0$, which implies $(c)$.
\end{proof}

\section{The action of the Galois group of $\Phi_n$}\label{galois_action_section}

We continue using here the notation introduced in the previous section. The genus computations in \S\S\ref{computation_section}-\ref{bounds_section}, which form the core of this article, rely fundamentally on Propositions \ref{unramified_valuation_count} and \ref{dynatomic_inertia_prop}. In order to apply these propositions effectively, we require a precise understanding of the elements of $G$ whose cycle decompositions have the forms described in Proposition \ref{dynatomic_inertia_prop}. In addition, explicit formulas for the orders of the centralizers of these elements will be needed when applying Proposition \ref{unramified_valuation_count}. The purpose of this section is to carry out a detailed analysis of the action of $G$ on the roots of $\Phi_n$. In the process we address both of the above requirements, the key result being Proposition \ref{conjugacy_centralizer_megaprop}.

Recall the notion of an isomorphism of group actions: if $A$ and $B$ are groups acting on sets $X$ and $Y$, respectively, we write $A\equiv B$ if there exist a group isomorphism $\varphi:A\to B$ and a bijection $\epsilon:X\to Y$ such that $\epsilon(ax)=\varphi(a)\epsilon(x)$ for all $a\in A$ and $x\in X$. Though the notation $A\equiv B$ does not make reference to the sets $X$ and $Y$, this should cause no confusion here because the sets being acted on will be clear from context. 

Let $R$ be the set of roots of $\Phi_n$ in the splitting field $N$, and consider the natural action of $G$ on $R$. In this section we will discuss three group actions, which we refer to as \emph{realizations} of $G$, that are isomorphic to $G$ with its action on $R$. The first realization is the automorphism group of a graph acting on its set of vertices; this is helpful as a visual aid for understanding the action of $G$. The second realization is a particular subgroup of the symmetric group $S_D$ acting on the set $\{1,\ldots, D\}$; this is useful for carrying out explicit computations with elements of $G$. The third realization is a wreath product $(\Z/n\Z)\wr S_r$ acting on the set $(\Z/n\Z)\times\{1,\ldots, r\}$. Though somewhat more technical, we find that this realization is the most convenient for purposes of proving the main results of this section. The key fact needed to show that these realizations are isomorphic is a well-known theorem of Bousch, namely Theorem 3 in \cite[Chap. 3]{bousch}.

\subsection{The group $G$ as a graph automorphism group}\label{graph_group_section} 

It is a simple consequence of the definition of $\Phi_n$ that the map $\phi(x)=x^2+t$ permutes the elements of $R$ (see \cite[\S2.2]{krumm_lgp} for details). Regarding $\phi$ as an element of the symmetric group $\Sym(R)$, we may therefore partition the set $R$ into $\phi$-orbits. By \cite[Thm. 2.4(c)]{morton-patel}, the fact that $\Phi_n$ is irreducible implies that every orbit has size $n$; hence, the number of orbits is $(\#R)/n=D/n=r$.

Let $\calG$ be the natural embedding of $G$ in $\Sym(R)$, and note that $G\equiv\calG$. Let $\Gamma$ be the directed graph whose vertices are the elements of $R$ and which has an edge $x\to\phi(x)$ for every $x\in R$. An illustration of $\Gamma$ is shown in Figure \ref{graph_figure} below. By Bousch's theorem, $\calG$ is the centralizer of $\phi$ in $\Sym(R)$. (More explicitly, this is a consequence of the proof of Proposition \ref{base_change_prop}. In the notation of that proof, we have $\calG=G_{\F}$, where $\F=\bar\Q$.) It follows that $\calG=\Aut(\Gamma)$ and therefore $G\equiv\Aut(\Gamma)$.

\begin{figure}[b]
\centerline{\includegraphics[width=5in]{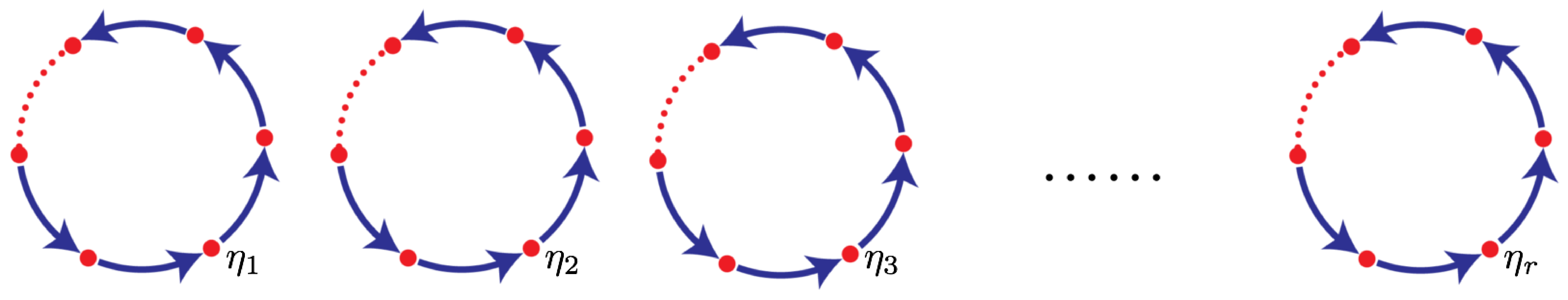}}
\caption{A directed graph whose automorphism group is isomorphic to the Galois group of $\Phi_n$. Every cycle in the graph has $n$ vertices, and there are $r$ cycles in total.}
\label{graph_figure}
\end{figure}

\subsection{The group $G$ as a permutation group}\label{permutation_group_section}

Let $S_D$ be the symmetric group on the set $\{1,\ldots, D\}$ and let $\sigma\in S_D$ be the permutation defined by 
\[\sigma=(1,\ldots, n)(n+1,\ldots, 2n)\cdots(D-n+1, \ldots, D).\]
There is a bijection $\ell:\{1,\ldots, D\}\to R$ under which the cycles in the decomposition of $\sigma$ correspond to the cycles in the graph $\Gamma$. Indeed, if we choose representatives $\eta_1,\ldots, \eta_r$ of the distinct cycles in $\Gamma$, then one such map $\ell$ is given by
\[\ell(ni-j)=\phi^{n-j}(\eta_i)\;\;\text{ for }1\le i\le r\text{ and }0\le j<n.\]
The map $\ell$ induces an isomorphism $\iota:S_D\to\Sym(R)$ under which $\sigma$ maps to $\phi$. Let $\ZZ$ be the centralizer of $\sigma$ in $S_D$. Since $\calG$ is the centralizer of $\phi$ in $\Sym(R)$, the image of $\ZZ$ under $\iota$ is equal to $\calG$. Moreover, the maps $\iota$ and $\ell$ induce an isomorphism of group actions between $\ZZ$ and $\calG$; hence $G\equiv\ZZ$.

\subsection{Background on wreath products}

Before discussing the realization of $G$ as a wreath product, we recall the basic construction of wreath products. For further information on this topic we refer the reader to \cite[\S2.6]{dixon-mortimer}, \cite[Chap. 7]{rotman}, or \cite[Chap. I]{kerber}.

Let $S_r$ denote the symmetric group on the set $\Omega=\{1,\ldots, r\}$. Let $A$ be a group, and consider the direct product $A^r$ consisting of functions $f:\Omega\to A$ with pointwise multiplication. There is an action of $S_r$ on $A^r$ given by $\pi\cdot f=f_{\pi}$, where $f_{\pi}$ is the function
\[f_{\pi}(i)=f(\pi^{-1}(i))\;\;\text{ for every }i\in\Omega.\]

This action induces a homomorphism $S_r\to\Aut(A^r)$, so we may form the semidirect product $\WW=A^r\rtimes S_r$. Elements of $\WW$ have the form $(f,\pi)$, where $f\in A^r$ and $\pi\in S_r$; the group operation in $\WW$ is given by
\[(f,\pi)(g,\sigma)=(fg_{\pi}, \pi\sigma).\]

The group $\WW$ is the wreath product of $A$ with $S_r$, denoted $A\wr S_r$. Letting $e$ and $1$, respectively, denote the identity elements of $A^r$ and $S_r$, there are embeddings $A^r\into\WW$ and $S_r\into\WW$ given by $f\mapsto (f,1)$ and $\pi\mapsto (e,\pi)$; we will henceforth identify $A^r$ and $S_r$ with their images under these maps. The group $B=A^r$, called the \emph{base group} of the wreath product, is a normal subgroup of $\WW$; indeed, $B$ is the kernel of the projection map $\WW\to S_r$ given by $(f,\pi)\mapsto\pi$. Furthermore, $S_r$ is a complement for $B$ in the sense that $B\cap S_r$ is trivial and $BS_r=\WW$.

Suppose now that $A$ acts on a set $\Delta$. Then there is an action of $\WW$ on the Cartesian product $\Delta\times\Omega$ given by
\begin{equation}\label{wreath_action}(f,\pi)\cdot(d,i)=\left(f(\pi(i))\cdot d,\pi(i)\right).
\end{equation}
Moreover, this action is faithful if $A$ acts faithfully on $\Delta$.

\subsection{The group $G$ as a wreath product}\label{wreath_presentation_section}

For the remainder of this section we assume that $A=\Z/n\Z$, so that $\WW=(\Z/n\Z)\wr S_r$. The action of $A$ on itself by addition induces a faithful action of $\WW$ on the set $X=A\times\Omega$ given by \eqref{wreath_action}. We will show that $\WW\equiv G$.

Let $\eta_1,\ldots, \eta_r$ be representatives of the distinct $\phi$-orbits of $R$. For every $w=(f,\pi)\in\WW$ we define $\zeta_w\in\calG=\Aut(\Gamma)$ by
\[\zeta_w(\phi^a(\eta_i))=\phi^{f(\pi(i))+a}(\eta_{\pi(i)})\;\;\text{for }a\in A\;\text{and }i\in\Omega.\]
Note that the notation $\phi^a$ for $a\in A$ is unambiguous since $\phi^n$ is the identity element of $\Sym(R)$.
Using the fact that $\calG$ is the centralizer of $\phi$ in $\Sym(R)$, it is a simple exercise to show that $\zeta_w$ is a well-defined element of $\calG$, and that the map $\zeta:\WW\to\calG$ given by $w\mapsto\zeta_w$ is a group isomorphism.

Let $\epsilon:X\to R$ be the map defined by $\epsilon(a, i)=\phi^a(\eta_i)$. From the definitions it follows that $\epsilon$ is a bijection and that for every $w\in\WW$ and $\alpha\in X$ we have $\epsilon(w\alpha)=\zeta(w)\epsilon(\alpha)$. Hence $\WW\equiv\calG$, and therefore $G\equiv\WW$. Using this realization of $G$ as a wreath product, we will now study the action of $G$.

\begin{rem}
It follows from the above discussion that
\[\Aut(\Gamma)\cong(\Z/n\Z)\wr S_r.\]
This is a special case of a well-known theorem of Frucht in graph theory. As shown in \cite{frucht} (see also  \cite[Thm. 14.5]{harary}), if $\Lambda$ is a finite connected graph and $\Gamma$ is a graph consisting of $r$ disjoint copies of $\Lambda$, then $\Aut(\Gamma)\cong\Aut(\Lambda)\wr S_r$.
\end{rem}

\subsection{Conjugacy in $\WW$}

Our main reason for using the realization of $G$ as a wreath product is that it provides convenient ways of deciding whether two elements of $G$ are conjugates of each other, and of calculating the order of the centralizer of any element of $G$. The key notion needed for these tasks is the \emph{type} of an element of $\WW$, defined below.

For every cycle $C=(i_1,i_2,\ldots, i_k)\in S_r$ and every element $f\in A^r$, we denote by $f(C)$ the element of $A$ given by $f(C)=f(i_1)+\cdots+f(i_k)$.

For every element $w=(f,\pi)\in\WW$, we define a map $T_w:X\to\Z_{\ge 0}$ as follows: for $a\in A$ and $k\in\Omega$, $T_w(a,k)$ is the number of $k$-cycles $C$ in the cycle decomposition of $\pi$ such that $f(C)=a$. The map $T_w$ will be called the \emph{type} of $w$. When $w$ is clear from context, we will denote $T_w(a,k)$ simply by $t_{ak}$ and we will use matrix notation $(t_{ak})$ to denote the map $T_w$.

\begin{prop}\label{wreath_conjugacy_centralizer}\mbox{}
\begin{enumerate}
\item Let $w_1,w_2\in\WW$. Then $w_1$ and $w_2$ are conjugates if and only if they have the same the type.
\item If $w$ has type $(t_{ak})$, then the order of the centralizer of $w$ in $\WW$ is given by the formula
\[\prod_{a\in A}\prod_{k\in\Omega}(t_{ak})!(kn)^{t_{ak}}.\]
\end{enumerate}
\end{prop}
\begin{proof}
Both statements can be deduced from more general results proved in \cite{kerber}. Specifically, (1) follows from item 3.7 on page 44, and (2) follows from item 3.9 on page 47.
\end{proof}

\subsection{The action of $\WW$}\label{wreath_action_section}

In this section we prove various properties of the action of $\WW$ on $X$. For elements $w=(f,\pi)\in\WW$ and $\alpha=(a,i)\in X$, we will denote by $w(\alpha)$ the action of $w$ on $\alpha$. Thus,
\begin{equation}\label{wreath_action_eq}
w(\alpha)=(f(\pi(i))+a,\pi(i)).
\end{equation}

Let $C_i=A\times\{i\}$ for $1\le i\le r$. Under the map $\epsilon$ defined in \S\ref{wreath_presentation_section}, $C_i$ corresponds to the $i^{\text{th}}$ cycle in the graph $\Gamma$, i.e., the cycle containing $\eta_i$.

The base group $A^r\le\WW$ is generated by the elements $\rho_1,\ldots,\rho_r$ defined by $\rho_i=(\delta_i,1)$, where $\delta_i(j)=0$ if $j\ne i$ and $\delta_i(i)=1$. Note that $\rho_i$ maps $C_i$ to itself and acts as the identity on $C_j$ if $j\ne i$. Viewed as an element of $\Aut(\Gamma)$ (via the map $\zeta$ defined in \S\ref{wreath_presentation_section}), $\rho_i$ acts as a $1/n$ rotation on the $i^{\text{th}}$ cycle. Let $\rho=\rho_1\cdots\rho_r=(\delta,1)$, where $\delta(i)=1$ for all $i\in\Omega$. Then $\zeta(\rho)=\phi$, so $\rho$ is in the center of $\WW$. A simple calculation shows that for all $s\in\Z$, $a\in A$, and $i\in\Omega$ we have
\begin{equation}\label{rotation_power_action}
\rho^s(a,i)=\rho_i^s(a,i)=(a+\bar s, i).
\end{equation}

For every $w\in\WW$ and every $i\in\Omega$, let $w(C_i)=\{w(\alpha)\;\vert\;\alpha\in C_i\}$.
\begin{lem}\label{w_cycle_action_lem} Let $w=(f,\pi)\in\WW$ and let $i\in\Omega$.
\begin{enumerate}
\item Letting $j=\pi(i)$, we have $w(C_i)=C_j$.
\item If $w(C_i)=C_i$, then there exists $0\le s<n$ such that $w(\alpha)=\rho_i^s(\alpha)$ for every $\alpha\in C_i$. Moreover, the $w$-orbit of every element of $C_i$ has cardinality $n/\gcd(n,s)$.
\end{enumerate}
\end{lem}
\begin{proof}
For every element $(a,i)\in C_i$ we have $w(a,i)=(f(j)+a,j)\in C_j$, so $w(C_i)\subseteq C_j$. Since $\#C_i=\#C_j$ and $w$ acts as a bijection on $X$, this implies that $w(C_i)=C_j$, proving (1). Suppose now that $w(C_i)=C_i$, and let $0\le s<n$ be such that $w(0,i)=(\bar s,i)$. By \eqref{rotation_power_action} we have $w(0,i)=\rho^s(0,i)$. Given $\alpha\in C_i$, we may write $\alpha$ in the form $\alpha=(\bar k, i)=\rho^k(0,i)$ for some integer $k$. Using \eqref{rotation_power_action} and the fact that $w$ commutes with $\rho$ we obtain
\[w(\alpha)=w\rho^k(0,i)=\rho^kw(0,i)=\rho^k\rho^s(0,i)=\rho^s\rho^k(0,i)=\rho^s(\alpha)=\rho_i^s(\alpha).\]
This proves the first statement in (2). Since $w$ acts like $\rho_i^s$ on $C_i$, the orbit of $\alpha$ under $w$ is equal to its orbit under $\rho_i^s$. The cyclic group generated by $\rho_i^s$ has order $n/\gcd(n,s)$, and it follows from \eqref{rotation_power_action} that the stabilizer of $\alpha$ in this group is trivial; hence the orbit of $\alpha$ has cardinality $n/\gcd(n,s)$. This completes the proof of (2).
\end{proof}

\begin{lem}\label{period2_swapping_lem}
Let $w=(f,\pi)\in\WW$. Suppose that $i, j\in\Omega$ are such that $w(C_i)=C_j, w(C_j)=C_i$, and $w^2(\alpha)=\alpha$ for every $\alpha\in C_i\cup C_j$. Then there exists $0\le s<n$ such that $w(\alpha)=\rho_i^{-s}\pi\rho_i^s(\alpha)$ for every $\alpha\in C_i\cup C_j$.
\end{lem}
\begin{proof}
Let $w(0,i)=(\bar s, j)$ and $w(0,j)=(\bar t, i)$ with $0\le s,t<n$. From \eqref{rotation_power_action} and the fact that $w$ commutes with $\rho$ it follows that for every integer $k$ we have $w(\bar k, i)=(\bar s+\bar k, j)$ and $w(\bar k,j)=(\bar t+\bar k, i)$.
Using this we calculate $w^2(0,i)=w(\bar s, j)=(\bar t+\bar s, i)$. Since $w^2(0,i)=(0,i)$, this implies that $\bar t=-\bar s$; thus, for every integer $k$ we have
\begin{equation}\label{double_rotation_action}
w(\bar k, i)=(\bar k+\bar s, j)\;\;\text{and}\;\;w(\bar k,j)=(\bar k-\bar s, i).
\end{equation}

Since $w(C_i)=C_j$ and $w(C_j)=C_i$, Lemma \ref{w_cycle_action_lem} implies that $\pi(i)=j$ and $\pi(j)=i$. It follows that for every $a\in A$ we have $\pi(a,i)=(a,j)$ and $\pi(a,j)=(a,i)$. Let $w'=\rho_i^{-s}\pi\rho_i^s$. If $\alpha=(\bar k, i)\in C_i$, then a simple calculation shows that $w'(\alpha)=(\bar k+\bar s, j)$,
so $w'(\alpha)=w(\alpha)$ by \eqref{double_rotation_action}. Similarly, if $\alpha=(\bar k, j)\in C_j$, then $w'(\alpha)=(\bar k-\bar s, i)=w(\alpha)$. Therefore $w(\alpha)=\rho_i^{-s}\pi\rho_i^s(\alpha)$ for every $\alpha\in C_i\cup C_j$.
\end{proof}

We can now prove the main result of this section.

\begin{prop}\label{conjugacy_centralizer_megaprop} Let $w\in\WW$ and let $\calC$ be the centralizer of $w$ in $\WW$.
\begin{enumerate}[leftmargin=7mm,itemsep=1.5mm]
\item Suppose that $w$ has cycle type $(2,D/2)$. Then the following hold.
\vspace{1mm}
\begin{enumerate}[leftmargin=4mm, itemsep=1.5mm]
\item[(a)]
Assume $w(C_i)\ne C_i$ for all $i\in\Omega$. Then $r$ is even, $w$ is conjugate to the permutation $(1,2)(3,4)\cdots(r-1,r)\in S_r$, and $|\calC|=(r/2)!(2n)^{r/2}$.
\item[(b)] Assume $w(C_i)=C_i$ for some $i\in\Omega$. Then $n$ is even and there exists $0<\ell\le r$ such that $r-\ell$ is even and $w$ is conjugate to the element $(\rho_1\cdots\rho_{\ell})^{n/2}\epsilon$, where $\epsilon=(\ell+1,\ell+2)\cdots(r-1,r)\in S_r$. Moreover, we have $|\calC|=\ell!((r-\ell)/2)!n^{\ell}(2n)^{(r-\ell)/2}$.
\end{enumerate}
\item Suppose that $w$ has cycle type $(2,n)$. Then the following hold.
\vspace{1mm}
\begin{enumerate}[leftmargin=4mm,itemsep=1.5mm]
\item Assume $w(C_i)=C_i$ for all $i\in\Omega$. Then $n$ is even, there exist indices $i<j\in\Omega$ such that $w=(\rho_i\rho_j)^{n/2}$, and $|\calC|=2(r-2)!n^r$.
\item Assume $w(C_i)\ne C_i$ for some $i\in\Omega$. Then there exist indices $i<j\in\Omega$ and an integer $0\le s<n$ such that $w=\rho_i^{-s}\tau\rho_i^s$, where $\tau=(i, j)\in S_r$. In this case, $|\calC|=2(r-2)!n^{r-1}$.
\end{enumerate} 
\item Suppose that $w$ moves exactly $n$ elements of $X$. Then $w=\rho_i^s$ for some $i\in\Omega$ and some integer $0<s<n$. Moreover, $|\calC|=(r-1)!n^r$.
\end{enumerate}
\end{prop}
\begin{proof}
Let $f\in A^r$ and $\pi\in S_r$ be such that $w=(f,\pi)$, and let $(t_{ak})$ be the type of $w$. We begin by proving 1(a). The hypothesis in (1) together with the fact that $\WW$ acts faithfully on $X$ imply that $w^2=(f+f_{\pi},\pi^2)$ is the identity element $(e,1)$; in particular, $\pi^2=1$. Moreover, by Lemma \ref{w_cycle_action_lem} we have $w(C_i)=C_{\pi(i)}$ for every $i\in\Omega$, so $\pi(i)\ne i$ for every $i$. Hence the $\pi$-orbit of every element of $\Omega$ has cardinality 2. It follows that $r$ is even, say $r=2m$, and $\pi$ is a product of $m$ disjoint transpositions. We can now determine the type of $w$.

Let $\{i_1,\pi(i_1)\},\ldots, \{i_m,\pi(i_m)\}$ be the orbits of $\pi$. Since $\pi$ has no $k$-cycles if $k=1$ or $k>2$, then $t_{ak}=0$ for all such $k$. When $k=2$, $t_{ak}$ is the number of indices $1\le v\le m$ such that $f(i_v)+f(\pi(i_v))=a$. Since $\pi^2=1$, this is equivalent to $f(i_v)+f_{\pi}(i_v)=a$. Now, as mentioned above, $w^2=(f+f_{\pi},\pi^2)=(e,1)$, so $f+f_{\pi}=e$ and therefore $f(i)+f_{\pi}(i)=0$ for every $i\in\Omega$. Hence, the condition $f(i_v)+f_{\pi}(i_v)=a$ is equivalent to $a=0$. Thus we have $t_{a2}=0$ if $a\ne 0$, and $t_{02}=m$. This determines the type of $w$. It is now trivial to check that $w$ has the same type as the permutation $\tau=(1,2)(3,4)\cdots(r-1,r)\in S_r$. It follows from Proposition \ref{wreath_conjugacy_centralizer} that $w$ is conjugate to $\tau$ and that $|\calC|=m!(2n)^{m}$; this completes the proof of 1(a).

Next we prove 1(b). Suppose that $i\in\Omega$ satisfies $w(C_i)=C_i$. By Lemma \ref{w_cycle_action_lem}, there exists $0\le s<n$ such that $w$ acts like $\rho_i^s$ on $C_i$, and the $w$-orbit of every element of $C_i$ has cardinality $n/\gcd(n,s)$. By hypothesis every orbit has size 2, so $n/\gcd(n,s)=2$, and hence $n$ must be even and $s=n/2$.

Let $i_1,\ldots, i_{\ell}$ be all the indices $i$ in $\Omega$ such that $w(C_i)=C_i$. Clearly, $0<\ell\le r$. Arguing as in the proof of 1(a), we see that $\pi$ fixes $i_k$ for each $k$, and that if $i\in\Omega\setminus\{i_1,\ldots, i_{\ell}\}$, then the orbit of $i$ under $\pi$ has size 2. This implies that $r-\ell$ is even, say $r-\ell=2q$, and the disjoint cycle decomposition of $\pi$ is a product of $\ell$ 1-cycles and $q$ transpositions. The type of $w$ is now easy to determine as done in case 1(a).

Clearly, $t_{ak}=0$ if $k>2$. Let $\{i_1\},\ldots, \{i_{\ell}\}, \{j_1,\pi(j_1)\},\ldots,\{j_q,\pi(j_q)\}$ be the orbits of $\pi$. Then $t_{a2}$ is the number of indices $1\le v\le q$ such that $f(j_v)+f_{\pi}(j_v)=a$. But $f+f_{\pi}=e$, so $t_{a2}=0$ if $a\ne 0$, and $t_{02}=q$. To determine $t_{a1}$ we need an additional observation. We know that for every index $1\le v\le\ell$, $w$ acts like $\rho_{i_v}^s$ on $C_{i_v}$. In particular, by \eqref{rotation_power_action} we have $w(0,i_v)=(\bar s, i_v)$. However, by \eqref{wreath_action_eq}, $w(0,i_v)=(f(i_v),i_v)$. Thus $f(i_v)=\bar s$ for all $v$. Now, $t_{a1}$ is the number of indices $1\le v\le\ell$ such that $f(i_v)=a$. Clearly then, $t_{a1}=0$ if $a\ne \bar s$ and $t_{\bar s1}=\ell$. This determines the type of $w$.

Proposition \ref{wreath_conjugacy_centralizer} yields $|\calC|=\ell!q!n^{\ell}(2n)^q$. Let $w'=(\rho_1\cdots\rho_{\ell})^s\epsilon$, where $\epsilon=(\ell+1,\ell+2)\cdots(r-1,r)\in S_r$. A straightforward calculation shows that $w'$ has the same type as $w$, and is therefore conjugate to $w$. This completes the proof of 1(b).

We now prove 2(a). If $w$ acts nontrivially on $m$ of the sets $C_i$, then the number of elements moved by $w$ is $mn$; hence $m=2$, so $w$ acts trivially on all but two of these sets, say $C_i$ and $C_j$ with $i<j$. By Lemma \ref{w_cycle_action_lem}, there exist integers $0<u,v<n$ such that $w$ acts like $\rho_i^u$ on $C_i$ and like $\rho_j^v$ on $C_j$. The $w$-orbit of every element of $C_i$ then has size $n/\gcd(n,u)=2$, so $n$ is even and $u=n/2$. Similarly, $v=n/2$. Thus $w$ acts like $(\rho_i\rho_j)^{n/2}$ on all of $X$, and therefore $w=(\rho_i\rho_j)^{n/2}$. Letting $s=n/2$, we have $w=(s\delta_i+s\delta_j,1)$; the type of $w$ is now easily determined. 

We have $t_{ak}=0$ if $k>1$, and $t_{a1}$ is the number of indices $k\in\Omega$ such that $s\delta_i(k)+s\delta_j(k)=a$. Now, note that $s\delta_i(k)+s\delta_j(k)=0$ if $k\ne i, j$, and $s\delta_i(k)+s\delta_j(k)=\bar s$ if $k=i$ or $j$. Hence $t_{a1}=0$ if $a\notin\{0,\bar s\}$, $t_{\bar s1}=2$, and $t_{01}=r-2$. Proposition \ref{wreath_conjugacy_centralizer} now yields $|\calC|=2(r-2)!n^r$; this proves 2(a).

Next we prove 2(b). By Lemma \ref{w_cycle_action_lem} we have $w(C_i)=C_j$ for some $j\ne i$. Then $w(C_j)$ must equal $C_i$, for otherwise $w$ would move more than $2n$ elements of $X$. Thus $w(C_i)=C_j, w(C_j)=C_i$, and $w$ acts trivially on $C_k$ for all $k\ne i, j$. It follows from Lemma \ref{w_cycle_action_lem} that $\pi=(i, j)$. Reversing the roles of $i$ and $j$ if necessary, we may assume that $i<j$. By Lemma \ref{period2_swapping_lem}, there exists $0\le s<n$ such that $w(\alpha)=\rho_i^{-s}\pi\rho_i^s(\alpha)$ for every $\alpha\in C_i\cup C_j$. Clearly, this equality also holds if $\alpha\in C_k$ with $k\notin\{i, j\}$, so $w=\rho_i^{-s}\pi\rho_i^s$. We can now determine the type of $w$. 

Since $w$ is conjugate to $\pi=(i,j)$, then $w$ and $\pi$ have the same type. We thus find that $t_{ak}=0$ if $k>2$; $t_{a2}=0$ if $a\ne 0$, and $t_{02}=1$; $t_{a1}=0$ if $a\ne 0$, and $t_{01}=r-2$. Proposition \ref{wreath_conjugacy_centralizer} now yields $|\calC|=2(r-2)!n^{r-1}$; this completes the proof of 2(b).

Finally, we prove (3). It is easy to see that the $n$ elements moved by $w$ must form one of the sets $C_i$. This implies that $w(C_i)=C_i$ and $w$ acts trivially on $C_j$ for all $j\ne i$. By Lemma \ref{w_cycle_action_lem}, there exists $0<s<n$ such that $w(\alpha)=\rho_i^s(\alpha)$ for every $\alpha\in C_i$. This equality clearly holds for $\alpha\notin C_i$ as well, so $w=\rho_i^s$. Using the relation $w=\rho_i^s=(s\delta_i,1)$, it is now a simple calculation to show that $t_{ak}=0$ if $k>1$, $t_{a1}=0$ if $a\notin\{0,\bar s\}$, $t_{\bar s1}=1$, and $t_{01}=r-1$. Proposition \ref{wreath_conjugacy_centralizer} now yields $|\calC|=(r-1)!n^r$.
\end{proof}

Having developed all of the necessary tools, we proceed to prove the main results of this article.

\section{Genus computations for $n=5$ and 6}\label{computation_section}

Recall the following notation from \S\ref{dynatomic_section}: $K=\bar\Q(t)$, $N/K$ is a splitting field of $\Phi_n$, $G=\Gal(N/K)$, $F$ is a subfield of $N$ obtained by adjoining one root of $\Phi_n$ to $K$, and $\PP=\{p_{\infty}\}\cup\{p_b\;\vert\;b\in\cup_{d|n}R_{n,d}\}$ is the set of places of $K$ that ramify in $F$. Finally, for any intermediate field $L$ in the extension $N/K$ and any place $p$ of $K$, $\PP_L(p)$ denotes the set of places of $L$ lying over $p$.

We begin this section by discussing an approach to the problem of computing the genera of subextensions of $N/K$. Let $H$ be a subgroup of $G$ with fixed field $L$, and let $g(L)$ denote the genus of $L$. We claim that if $p$ is a place of $K$ which ramifies in $L$, then $p\in\PP$. Indeed, if $p$ ramifies in $L$, then it ramifies in $N$. Letting $\P$ be a place of $N$ lying over $p$, the inertia group $I_{\P|p}$ is nontrivial, so Corollary \ref{function_field_inertia_cor} implies that $p$ ramifies in $F$. Hence $p\in\PP$.

The Hurwitz genus formula \cite[Cor. 3.5.6]{stichtenoth} now yields

\begin{equation}\label{hurwitz_general}
2g(L)-2=(-2)|G:H|+\sum_{p\in\PP}\sum_{\q\in\PP_L(p)}(e_{\q|p}-1).
\end{equation}
Let us define \[g_{n,\infty}(H)=\sum_{\q\in\PP_L(p_{\infty})}(e_{\q|p_{\infty}}-1),\]
and for every divisor $d$ of $n$,
\[g_{n,d}(H)=\sum_{b\in R_{n,d}}\sum_{\q\in\PP_L(p_b)}(e_{\q|p_b}-1).\]
By \eqref{hurwitz_general} we have the following expression for the genus of $L$:

\begin{equation}\label{genus_formula}
g(L)=1 - |G:H| + \frac{1}{2}\left(g_{n,\infty}(H) + \sum_{d|n}g_{n,d}(H)\right).
\end{equation}

The problem of computing $g(L)$ is thus reduced to the following: given any place $p\in\PP$, compute  the ramification index $e_{\q|p}$ for every $\q\in\PP_L(p)$. Our method for doing this is based on the following lemma. 

\begin{lem}\label{double_coset_ramification_lem}
Let $p\in\PP$, $\P\in\PP_N(p)$, and $I=I_{\P|p}$. Let $\sigma_1,\ldots,\sigma_m$ be representatives of the distinct double cosets in $I\backslash G/H$. Then
\[\{e_{\q|p}: \q \in\PP_L(p)\}=\{|I^{\sigma_i}:I^{\sigma_i}\cap H|:1\le i\le m\}.\]
\end{lem}
\begin{proof}
Since $K$ is a function field over $\bar\Q$, we have $f_{\P|p}=1$ and therefore $D_{\P|p}=I_{\P|p}=I$. Using Lemma \ref{double_coset_bijection_lem} we see that the set $\PP_L(p)$ consists of the places $\sigma_i(\P)\cap L$; moreover, if $\q=\sigma_i(\P)\cap L$, then $e_{\q|p}=|I^{\sigma_i}:I^{\sigma_i}\cap H|$. The result follows immediately.
\end{proof}

For purposes of explicit computation it is convenient to use the isomorphisms $G\equiv\WW\equiv\ZZ$ proved in \S\S\ref{permutation_group_section}-\ref{wreath_presentation_section}. With notation as in Lemma \ref{double_coset_ramification_lem}, suppose that one is able to identify the subgroup of $\WW$ (or $\ZZ$) which corresponds to the inertia group $I$. It is then a finite computation to determine representatives $\sigma_1,\ldots,\sigma_m$ and to compute the indices $|I^{\sigma_i}:I^{\sigma_i}\cap H|$. Carrying out this calculation for every $p\in\PP$, one obtains all the data needed to determine the numbers $g_{n,\infty}$ and $g_{n,d}$, and hence the genus of $L$.

The remainder of this section is devoted to showing that when $n=5$ or 6 it is possible -- and computationally feasible -- to identify inertia groups $I_{\P|p}$ for every $p\in\PP$, and thus to compute the genus of any intermediate field in the extension $N/K$. In particular, this allows us to obtain the genera of the fixed fields of all the maximal subgroup of $G$, and by applying Proposition \ref{En_finiteness_prop}, to show that the sets $E_5$ and $E_6$ are finite.

In order to carry out all the necessary computations we have used version 2.23-1 of \textsc{Magma} \cite{magma} running on a MacBook Pro with a 2.7 GHz Intel Core i5 processor and 8 GB of memory. The interested reader can find the code for our computations in \cite{krumm_finiteness_code}. The code relies primarily on four intrinsic \textsc{Magma} functions: \texttt{WreathProduct}, \texttt{MaximalSubgroups}, \texttt{DoubleCosetRepresentatives}, and \texttt{meet}. The first function applied to $\Z/n\Z$ and $S_r$ constructs the group $\WW$ together with the natural embeddings $S_r\into\WW$ and $(\Z/n\Z)^r\into\WW$. (It should be noted, however, that internally $\WW$ is constructed as the group $\ZZ$.) Once $\WW$ is constructed, the second function can be used to obtain the maximal subgroups of $\WW$ up to conjugacy; the algorithm used is described in \cite{cannon-holt}. Given subgroups $I$ and $H$ of $\WW$, the third function computes representatives of the double cosets in $I\backslash \WW/H$. Finally, the fourth function can be used to compute the intersection of two subgroups of $\WW$; the algorithm uses a backtrack method described in \cite{leon}.

Throughout this section we use the following notation. For $1\le i\le r$ we let $\rho_i$ be the element of $\WW$ defined in \S\ref{wreath_action_section}. As an automorphism of the graph $\Gamma$, $\rho_i$ is a $1/n$ rotation of the $i^{\text{th}}$ cycle. As an element of the group $\ZZ$, $\rho_i$ is the $i^{\text{th}}$ cycle in the decomposition of the permutation $\sigma$ defined in \S\ref{permutation_group_section}. For distinct indices $1\le i,j\le r$ we let $\tau_{i,j}$ be the transposition $(i,j)\in S_r$ regarded as an element of $\WW$. As an automorphism of $\Gamma$, $\tau_{i,j}$ interchanges the $i^{\text{th}}$ and $j^{\text{th}}$ cycles without performing any rotations.

\begin{lem}\label{rotation_conjugacy}
The elements $\rho_1,\ldots, \rho_r$ are conjugate in $\WW$. Moreover, if $i,j,u,v\in\{1,\ldots, r\}$ with $i\ne j$ and $u\ne v$, then $\rho_i\rho_j$ is conjugate to $\rho_u\rho_v$.
\end{lem}
\begin{proof}
This follows from Proposition \ref{wreath_conjugacy_centralizer}. The type $(t_{ak})$ of $\rho_i$ is independent of $i$; indeed, we have $t_{ak}=0$ if $k>1$, $t_{a1}=0$ if $a\ne 0,1$, $t_{01}=r-1$, and $t_{11}=1$. Similarly, if $i\ne j$, then the type $(t_{ak})$ of $\rho_i\rho_j$ independent of $i$ and $j$: we have $t_{ak}=0$ if $k>1$, $t_{a1}=0$ if $a\ne0,1$, $t_{01}=r-2$, and $t_{11}=2$.
\end{proof}

\subsection{The case $n=5$}\label{5case_section}

The polynomial $\Phi_5$ has $D=2\nu(5)=30$ roots which can be partitioned into $r=D/5=6$ cycles. Hence, the graph $\Gamma$ consists of six $5$-cycles. The group $\WW$ is $(\Z/5\Z)\wr S_6$, so $|G|=5^66!=$ 11,250,000. The set of places of $K$ which ramify in $F$ is $\PP=\{p_{\infty}\}\cup \{p_b\;\vert\; b\in R_{5,5}\cup R_{5,1}\}$; using Lemma \ref{delta_roots_disjoint_lem} we obtain $\#R_{5,5}=11$ and $\#R_{5,1}=4$. We will henceforth identify $G$ and $\WW$ using the isomorphism $G\equiv\WW$, where $G$ acts on the roots of $\Phi_5$ and $\WW$ acts on the set $X=(\Z/5\Z)\times \{1,\ldots, 6\}$. 

We define three subgroups of $\WW$ by $A=\langle\tau_{1,2}\tau_{3,4}\tau_{5,6}\rangle, B=\langle\tau_{1,2}\rangle, C=\langle\rho_1\rangle$.

\begin{lem}\label{inertia_in_W_lem_5}
Up to conjugation, $A$ is the only subgroup of $\WW$ generated by an element with cycle type $(2,15)$; similarly, $B$ is uniquely determined by the cycle type $(2,5)$, and $C$ by the cycle type $(5,1)$.
\end{lem}
\begin{proof} 
Suppose that $\tilde A$ is a subgroup of $\WW$ generated by an element $w$ with cycle type $(2,15)$. We are then in the context of case 1 of Proposition \ref{conjugacy_centralizer_megaprop}. Moreover, since $n=5$ is odd, case 1(b) is ruled out. Hence, by case 1(a), $w$ is conjugate to $\tau_{1,2}\tau_{3,4}\tau_{5,6}$, and therefore $\tilde A$ is conjugate to $A$.

Now suppose that a subgroup $\tilde B$ is generated by an element $w$ with cycle type $(2,5)$. By case 2(b) of Proposition \ref{conjugacy_centralizer_megaprop}, $w$ is conjugate to $\tau_{i,j}$ for some indices $i, j$. Clearly the permutations $(i,j)$ and $(1,2)$ are conjugates in $S_6$, so $\tau_{i,j}$ is conjugate to $\tau_{1,2}$ and therefore $\tilde B$ is conjugate to $B$.

Finally, suppose that a subgroup $\tilde C$ is generated by an element $w$ with cycle type $(5,1)$. By case 3 of Proposition \ref{conjugacy_centralizer_megaprop}, we have $w=\rho_i^s$ for some $i$ and $0<s<5$. Note that $\langle \rho_i^s\rangle=\langle\rho_i\rangle$ since $|\rho_i|=5$. By Lemma \ref{rotation_conjugacy}, $w$ is conjugate to $\rho_1^s$, and therefore $\tilde C=\langle w\rangle$ is conjugate to $\langle\rho_1^s\rangle=\langle\rho_1\rangle=C$.
\end{proof}

\begin{lem}\label{inertia_identification_lem_5}\mbox{}
\begin{enumerate}
\item There exists $\P\in\PP_N(p_{\infty})$ such that $I_{\P|p_{\infty}}=A$.
\item For every $b\in R_{5,5}$ there exists $\P\in\PP_N(p_b)$ such that $I_{\P|p_b}=B$.
\item For every $b\in R_{5,1}$ there exists $\P\in\PP_N(p_b)$ such that $I_{\P|p_b}=C$.
\end{enumerate}
\end{lem}
\begin{proof}
Let $\P\in\PP_N(p_{\infty})$. By Proposition \ref{dynatomic_inertia_prop} we have $I_{\P|p_{\infty}}=\langle w\rangle$, where $w\in\WW$ has cycle type $(2,15)$. Thus, by Lemma \ref{inertia_in_W_lem_5}, $I_{\P|p_{\infty}}$ is conjugate to $A$. Replacing $\P$ by a conjugate place if necessary, we then have $I_{\P|p_{\infty}}=A$. This proves (1); the proofs of (2) and (3) are similar. 
\end{proof}

\begin{prop}\label{genus_computable_5}
Let $H$ be a subgroup of $\WW$ with fixed field $L$. Suppose that $\alpha_1,\ldots, \alpha_t$ are double coset representatives for $A\backslash\WW/H$; $\beta_1,\ldots, \beta_u$ are representatives for $B\backslash \WW/H$; and $\gamma_1,\ldots, \gamma_v$ are representatives for $C\backslash\WW/H$. Then the genus of $L$ is given by
\[g(L)=1 - |\WW:H| + \frac{1}{2}\left(g_{5,\infty}(H) + g_{5,5}(H) + g_{5,1}(H)\right),\]
where
\begin{align}
\label{g_5_infty_formula}g_{5,\infty}(H)&=\sum_{i=1}^t(|A^{\alpha_i}:A^{\alpha_i}\cap H|-1),\\
\label{g_5_5_formula}g_{5,5}(H) &= 11\cdot\sum_{i=1}^u(|B^{\beta_i}:B^{\beta_i}\cap H|-1),\;\text{and}\\
\label{g_5_1_formula}g_{5,1}(H) &= 4\cdot\sum_{i=1}^v(|C^{\gamma_i}:C^{\gamma_i}\cap H|-1).
\end{align}
\end{prop}
\begin{proof}
The formula for $g(L)$ follows from \eqref{genus_formula}. Let $p=p_{\infty}$. By Lemma \ref{inertia_identification_lem_5}, there exists $\P\in\PP_N(p)$ such that $I_{\P|p}=A$. By Lemma \ref{double_coset_ramification_lem} we have
\[\{e_{\q|p}: \q \in\PP_L(p)\}=\{|A^{\alpha_i}:A^{\alpha_i}\cap H|:1\le i\le t\},\]
which implies \eqref{g_5_infty_formula}. Now suppose that $b\in R_{5,5}$ and let $p=p_b$. By Lemma \ref{inertia_identification_lem_5}, there exists $\P\in\PP_N(p)$ such that $I_{\P|p}=B$. Thus, by Lemma \ref{double_coset_ramification_lem},
\[\{e_{\q|p}: \q \in\PP_L(p)\}=\{|B^{\beta_i}:B^{\beta_i}\cap H|:1\le i\le u\},\]
and therefore $\sum_{\q\in\PP_L(p)}(e_{\q|p}-1)=\sum_{i=1}^u(|B^{\beta_i}:B^{\beta_i}\cap H|-1)$.
Since the value of this sum is independent of $b$, and $\#R_{5,5}=11$, then
\[g_{5,5}(H)=\sum_{b\in R_{5,5}}\sum_{\q\in\PP_L(p_b)}(e_{\q|p_b}-1)=11\cdot\sum_{i=1}^u(|B^{\beta_i}:B^{\beta_i}\cap H|-1),\]
which proves \eqref{g_5_5_formula}. The proof of \eqref{g_5_1_formula} is similar.
\end{proof}

We can now begin to prove Theorem \ref{finiteness_theorem_intro}.

\begin{thm}\label{E5_finiteness_thm}
The set $E_5$ is finite.
\end{thm}
\begin{proof}
Computing representatives for the conjugacy classes of maximal subgroups of $\WW$, we obtain 8 subgroups which we denote by $M_1,\ldots, M_8$. The indices of these subgroups in $\WW$ are given, respectively, by
\[|\WW:M_i|:\;3125, 15, 15, 10, 6, 6, 5, 2.\]
Let $L_i$ be the fixed field of $M_i$. Fixing an index $i$, we may compute representatives for the double cosets in $A\backslash\WW/M_i$, $B\backslash\WW/M_i$, and $C\backslash\WW/M_i$. The genus of $L_i$ can then be obtained by applying Proposition \ref{genus_computable_5}. Carrying out these computations for $i=1,\ldots, 8$ we obtain, respectively, the genera
\[9526, 21, 11, 9, 2, 12, 4,5.\]
The result now follows from Proposition \ref{En_finiteness_prop}. The values of $g_{5,\infty}(M_i)$, $g_{5,5}(M_i)$, and $g_{5,1}(M_i)$ are shown in Table \ref{period_5_table}.
\end{proof}

\begin{table}[h!]
\centering
\begin{tabular}{|c|c|c|c|c|c|c|c|c|} 
 \hline
& $M_1$ & $M_2$ & $M_3$ & $M_4$ & $M_5$ & $M_6$ & $M_7$ & $M_8$\\
 \hline
$g_{5,\infty}$ & 1550 & 4 & 6 & 3 & 3 & 1 & 0 & 1\\
 \hline
 $g_{5,5}$ & 13750 & 66 & 44 & 33 & 11 & 33 & 0 & 11\\
 \hline
 $g_{5,1}$ & 10000 & 0 & 0 & 0 & 0 & 0 & 16 & 0\\
 \hline
\end{tabular}
\bigskip
\caption{Ramification data for the maximal subgroups of $\WW$.}
\label{period_5_table}
\end{table}

\subsection{The case $n=6$}\label{6case_section}

Our next objective is to show that the set $E_6$ is finite. The structure of the proof is similar to the case $n=5$, though the process of identifying the necessary inertia groups requires an additional step that was not present in that case.

The polynomial $\Phi_6$ has $D=2\nu(6)=54$ roots which can be partitioned into $r=D/6=9$ cycles. Hence, the graph $\Gamma$ consists of nine 6-cycles. The group $\WW$ is $(\Z/6\Z)\wr S_9$, so $|G|=6^99!=$ 3,656,994,324,480. The set of places of $K$ which ramify in $F$ is $\PP=\{p_{\infty}\}\cup \{p_b\;\vert\; b\in \cup_{d|6}R_{6,d}\}$. Using Lemma \ref{delta_roots_disjoint_lem} we find that
\begin{equation}\label{delta_degrees_6}
\#R_{6,6}=20,\;\#R_{6,3}=3,\;\#R_{6,2}=2,\;\#R_{6,1}=2.
\end{equation}

We define several cyclic subgroups of $\WW$. For $0\le j\le 4$, let
\[\gamma_j=\left(\prod_{i=1}^{9-2j}\rho_i^3\right)\left(\prod_{i=0}^{j-1}\tau_{8-2i,9-2i}\right)\;\;\text{and}\;\;A_j=\langle\gamma_j\rangle.\]
In addition, let $B_0=\langle\rho_1^3\rho_2^3\rangle,\;B_1= \langle\tau_{1,2}\rangle,\;C= \langle\rho_1^3\rangle,\;D= \langle\rho_1^2\rangle$, $E= \langle\rho_1\rangle$.

\begin{lem}\label{inertia_in_W_lem_6}
Up to conjugation, the groups $A_j$ are the only subgroups of $\WW$ generated by an element with cycle type $(2,27)$; $B_0$ and $B_1$ are the only subgroups generated by an element with cycle type $(2,6)$; and $C, D, E$ are uniquely determined by the cycle types $(2,3)$, $(3,2)$, and $(6,1)$, respectively.
\end{lem}
\begin{proof}
Suppose that $\tilde A=\langle w\rangle$, where $w\in\WW$ has cycle type $(2,27)$. We are then in the context of case 1 of Proposition \ref{conjugacy_centralizer_megaprop}. Moreover, since $r=9$ is odd, case 1(b) must hold. Thus, there exists $0<\ell\le9$ such that $9-\ell$ is even and $w$ is conjugate to $v=(\rho_1\cdots\rho_{\ell})^{3}(\tau_{\ell+1,\ell+2})\cdots\tau_{8,9}$. Writing $9-\ell=2j$ with $0\le j\le 4$, we have $v=\gamma_j$. Hence, $\tilde A$ is conjugate to $A_j$.

Suppose now that $\tilde B=\langle w\rangle$, where $w\in\WW$ has cycle type $(2,6)$. We are then in the context of case 2 of Proposition \ref{conjugacy_centralizer_megaprop}. In case 2(a) of the proposition, $w=(\rho_i\rho_j)^3$ for some indices $i\ne j$. By Lemma \ref{rotation_conjugacy}, this implies that $w$ is conjugate to $(\rho_1\rho_2)^3$, and therefore $\tilde B$ is conjugate to $B_0$. In case 2(b) of the proposition, $w$ is conjugate to $\tau_{1,2}$ and $\tilde B$ is conjugate to $B_1$.

We now prove the uniqueness of the group $C$ and omit the proofs for $D$ and $E$, which are similar. Suppose that $\tilde C=\langle w\rangle$, where $w\in\WW$ has cycle type $(2,3)$. By case (3) of Proposition \ref{conjugacy_centralizer_megaprop}, we have $w=\rho_i^s$ with $1\le i\le 9$ and $0<s<6$. In order for $w$ to have cycle type $(2,3)$ we must have $s=3$; thus, by Lemma \ref{rotation_conjugacy}, $w$ is conjugate to $\rho_1^3$ and $\tilde C$ is conjugate to $C$.
\end{proof}

Before continuing with the main discussion of this section, we prove a couple of auxiliary results. Returning to the general case of an arbitrary positive integer $n$, let $\theta$ be a root of $\Phi_n$ such that $F=K(\theta)$. Recall from \S\ref{dynatomic_section} that $F$ has an automorphism given by $\theta\mapsto\phi(\theta)$, and that $F_0$ denotes the fixed field of this automorphism. 

\begin{lem}\label{F0_primitive_element}
Let $\tau=\theta+\phi(\theta)+\cdots+\phi^{n-1}(\theta)$. Then $F_0=K(\tau)$.
\end{lem}
\begin{proof}
Following Morton \cite{morton_dynatomic_curves}, we define the \emph{trace} of a cycle in the graph $\Gamma$ to be the sum of the elements in the cycle. Note that $\tau$ is the trace of the cycle containing $\theta$. Let $P\in K[x]$ be the monic polynomial of degree $r$ whose roots are the traces of all the cycles in $\Gamma$. By \cite[Cor. 3, p. 335]{morton_dynatomic_curves}, $P$ is irreducible; hence $P$ is the minimal polynomial of $\tau$, and therefore $[K(\tau):K]=r$. Clearly $\tau$ is fixed by $\phi$, so $K(\tau)\subseteq F_0$. Now, since $[F:K]=D$ and $[F:F_0]=n$, then $[F_0:K]=D/n=r=[K(\tau):K]$. Thus $F_0=K(\tau)$.
\end{proof}

We can now describe the subgroup of $G$ corresponding to $F_0$.

\begin{lem}\label{F0_group_lem}
Let $\O=\{\theta,\phi(\theta),\ldots, \phi^{n-1}(\theta)\}$ and let $H_0$ be the setwise stabilizer of $\O$ in $G$. Then $F_0$ is the fixed field of $H_0$.
\end{lem} 

\begin{proof}
Let $U$ and $V$ be the subgroups of $G$ defined by 
\begin{align*}
U&=\{\sigma\in G\;\vert\;\sigma(x)=x\;\text{ for every}\;x\in \O\},\\
V&=\{\sigma\in G\;\vert\;\sigma(x)=x\;\text{ for every }\;x\in R\setminus\O\}.
\end{align*}
A simple argument shows that $H_0=UV$; see Example 2 in \cite[p. 172]{dummit-foote}.

The fact that $\phi$ is in the center of $G$ implies that $U$ is equal to the stabilizer of $\theta$ in $G$; thus $U=\Gal(N/F)$. It follows that $F$ is the fixed field of $U$. Let $L$ be the fixed field of $H_0$. Since $U\le H_0$, then $L\subseteq F$. Defining $\tau$ as in Lemma \ref{F0_primitive_element}, it is clear that $\tau$ is fixed by every element of $H_0$; hence $F_0=K(\tau)\subseteq L$. We have thus shown that $F_0\subseteq L\subseteq F$. To complete the proof we will show that $[F:L]=[F:F_0]$.

Identifying $G$ with $\Aut(\Gamma)$ we see that $V$ consists of the elements of $G$ that act trivially on every cycle of $\Gamma$ except possibly on the cycle containing $\theta$. Thus the elements of $V$ are the $n$ rotations of the latter cycle, so $|V|=n$. By Galois theory we have $[F:L]=|UV|/|U|=|V|$, where the second equality uses the fact that $U\cap V=\{1\}$. We conclude that $[F:L]=n=[F:F_0]$.
\end{proof}

We return now to the case $n=6$.

\begin{lem}\label{inertia_identification_lem_6}\mbox{}
\begin{enumerate}
\item There exists $\P\in\PP_N(p_{\infty})$ such that $I_{\P|p_{\infty}}=A_4$.
\item For every $b\in R_{6,6}$ there exists $\P\in\PP_N(p_b)$ such that $I_{\P|p_b}=B_1$.
\item For every $b\in R_{6,3}$ there exists $\P\in\PP_N(p_b)$ such that $I_{\P|p_b}=C$.
\item For every $b\in R_{6,2}$ there exists $\P\in\PP_N(p_b)$ such that $I_{\P|p_b}=D$.
\item For every $b\in R_{6,1}$ there exists $\P\in\PP_N(p_b)$ such that $I_{\P|p_b}=E$.
\end{enumerate}
\end{lem}

\begin{proof} 
Let $p=p_{\infty}$, $\P\in\PP_N(p)$, and $I=I_{\P|p}$. By Proposition \ref{dynatomic_inertia_prop}, $I$ has a generator with cycle type $(2,27)$.  By Lemma \ref{inertia_in_W_lem_6}, $I$ must be conjugate to one of the groups $A_j$. Replacing $\P$ by a conjugate ideal if necessary, we then have $I=A_j$ for some $j$. We claim that $I=A_4$.

To prove this we will use the number $S(p)$ defined in Proposition \ref{F0_ramification_prop}. By part (a) of the proposition, $S(p)=9-e_6=4$. We can calculate $S(p)$ in a different way by using the inertia group $I$ as follows. Let $H_0$ be the subgroup of $\WW$ defined in Lemma \ref{F0_group_lem}. Applying Lemma \ref{double_coset_ramification_lem} we see that
\[S(p)=\sum_{i=1}^m(|I^{\sigma_i}:I^{\sigma_i}\cap H_0|-1),\]
where $\sigma_1,\ldots, \sigma_m$ are double coset representatives for $I\backslash\WW/H_0$. Assuming that $I=A_0,A_1,A_2,A_3, A_4$, respectively, we compute representatives $\sigma_i$ and use the above formula to obtain $S(p)=0,1,2,3,4$. However, we know that $S(p)=4$, so necessarily $I=A_4$, as claimed. This proves (1). For the purposes of this computation, we identify $\WW$ with the group $\ZZ\le S_{54}$, so that $H_0$ is identified with the setwise stabilizer of the set $\{1,\ldots, 6\}$ in $\ZZ$. The code used for these computations is available in \cite{krumm_finiteness_code}.

Let $b\in R_{6,6}$, $p=p_b$, $\P\in\PP_N(p)$, and $I=I_{\P|p}$. By Proposition \ref{dynatomic_inertia_prop}, $I$ has a generator with cycle type $(2,6)$. By Lemma \ref{inertia_in_W_lem_6}, $I$ must be conjugate to either $B_0$ or $B_1$. Replacing $\P$ by a conjugate ideal if necessary, we then have $I=B_0$ or $B_1$. We know that $S(p)=1$ by part (b) of Proposition \ref{F0_ramification_prop}. Now, assuming that $I=B_0, B_1$, respectively, the above displayed formula yields $S(p)=0,1$; hence $I=B_1$. This proves (2).

Statements (3)-(5) follow easily from Proposition \ref{dynatomic_inertia_prop} and Lemma \ref{inertia_in_W_lem_6}.
\end{proof}

\begin{prop}\label{genus_computable_6} Let $H$ be a subgroup of $\WW$ with fixed field $L$. For every group $I\in\{A_4, B_1, C, D, E\}$ let
\[q_H(I)=\sum_{i=1}^m(|I^{\sigma_i}:I^{\sigma_i}\cap H|-1),\]
where $\sigma_1,\ldots, \sigma_m$ are representatives of all the double cosets in $I\backslash\WW/H$. Then the genus of $L$ is given by
\[g(L)=1 - |\WW:H| + \frac{1}{2}\left(q_H(A_4)+20q_H(B_1)+3q_H(C)+2q_H(D)+2q_H(E)\right).\]
\end{prop}
\begin{proof}

Let $p=p_{\infty}$. By Lemma \ref{inertia_identification_lem_6}, there exists $\P\in\PP_N(p)$ such that $I_{\P|p}=A_4$. Using Lemma \ref{double_coset_ramification_lem} we see that $g_{6,\infty}(H)=q_H(A_4)$. Now let $b\in R_{6,6}$, $p=p_b$, and let $\P\in\PP_N(p)$ satisfy $I_{\P|p}=B_1$. By Lemma \ref{double_coset_ramification_lem},
\[q_H(B_1)=\sum_{\q\in\PP_L(p)}(e_{\q|p}-1).\]
Since this holds for every $b\in R_{6,6}$, then \eqref{delta_degrees_6} yields $g_{6,6}(H)=20 q_H(B_1)$. By a similar argument we show that
\[g_{6,3}(H)=3q_H(C),\; g_{6,2}(H)=2q_H(D),\text{ and } g_{6,1}(H)=2q_H(E).\]
The stated formula for the genus of $L$ is now a consequence of \eqref{hurwitz_general}.
\end{proof}

We can now prove a second part of Theorem \ref{finiteness_theorem_intro}.
\begin{thm}
The set $E_6$ is finite.
\end{thm}
\begin{proof}
Computing representatives for the conjugacy classes of maximal subgroups of $\WW$, we obtain 11 subgroups which we denote by $M_1,\ldots, M_{11}$. The indices of these subgroups in $\WW$ are given, respectively, by
\[|\WW:M_i|:\;840, 280,256,126,84,36,9,3,2,2,2.\]
Let $L_i$ be the fixed field of $M_i$. Fixing an index $i$, we may compute the numbers $q_{M_i}(I)$ for $I\in\{A_4, B_1, C, D, E\}$. The genus of $L_i$ can then be obtained by applying Proposition \ref{genus_computable_6}. Carrying out these computations for $i=1,\ldots, 11$ we obtain, respectively, the genera
\[3569, 837, 765, 255, 147, 43, 4, 2, 12, 9, 2.\]
By Proposition \ref{En_finiteness_prop}, this implies that $E_6$ is finite. The values of $q_{M_i}(I)$ are shown in Table \ref{period_6_table}.
\end{proof}

\begin{table}[h!]
\centering
\begin{tabular}{|c|c|c|c|c|c|c|c|c|c|c|c|} 
 \hline
& $M_1$ & $M_2$ & $M_3$ & $M_4$ & $M_5$ & $M_6$ & $M_7$ & $M_8$ & $M_9$ & $M_{10}$ & $M_{11}$\\
 \hline
$q_{M_i}(A_4)$ & 416 & 132 & 120 & 60 & 40 & 16 & 4 & 0 & 1 & 0 & 1\\
 \hline
 $q_{M_i}(B_1)$ & 420 & 105 & 64 & 35 & 21 & 7 & 1 & 0 & 1 & 1 & 0\\
 \hline
 $q_{M_i}(C)$ & 0 & 0 & 128 & 0 & 0 & 0 & 0 & 0 &1 & 0 & 1\\
  \hline
 $q_{M_i}(D)$ & 0 & 0 & 0 & 0 & 0 & 0 & 0 & 2& 0 & 0 & 0\\
  \hline
  $q_{M_i}(E)$ & 0 & 0 & 128 & 0 & 0 & 0 & 0 & 2 & 1 & 0 & 1\\
 \hline
\end{tabular}
\bigskip
\caption{Ramification data for the maximal subgroups of $\WW$.}
\label{period_6_table}
\end{table}

\section{Genus bounds for $n>6$}\label{bounds_section}

The methods used in the previous section for $n=5$ and 6 can, in principle, be applied to higher values of $n$; however, there are computational limitations which make this impractical. Firstly, for $n>10$ there are issues of both memory and time which prevent us from computing the maximal subgroups of $\WW$. Thus, we are restricted to considering only $n=7,8,9,10$. Furthermore, even for these values of $n$ there are similar complications in the crucial step of computing double coset representatives. Hence, it would appear that our methods cannot be extended beyond $n=6$. However, a modification of the method will allow us to show that $E_7$ and $E_9$ are finite.

Recall that our main goal is to show that the genera of the function fields corresponding to maximal subgroups of $G$ are all greater than 1. In the cases $n=5,6$ we did this by calculating the exact values of these genera, although it would be sufficient to prove a lower bound greater than 1. In this section we will show that, as long as the maximal subgroups of $\WW$ can be computed, it is possible to obtain lower bounds for the required genera. In the cases $n=7,9$ these bounds will suffice to prove the desired result. Unfortunately, the bounds are not good enough when $n=8,10$; the difficulties are explained in \S\ref{even_bounds_section}. We keep here all of the notation introduced in earlier sections.

\begin{lem}\label{genus_contrib_lower_bound_lem} Let $H$ be a subgroup of $G$ with fixed field $L$, and $a=|G:H|$. Let $p$ be a place of $K$, let $\{\q_1,\ldots, \q_s\}=\PP_L(p)$, and $e_i=e_{\q_i|p}$. Suppose that $u$ is an upper bound for the number of indices $i$ such that $e_i=1$.
Then \[\sum_{i=1}^s (e_i-1)\ge\lceil a-\lfloor(u+a)/2\rfloor\rceil.\]
\end{lem}
\begin{proof}
Let $x$ be the number of indices $i$ such that $e_i=1$, and let $y=s-x$. Note that $a=e_1+\cdots+e_s\ge x+2y$. Since $x\le u$, this implies $x+y\le(u+a)/2$. Thus $s\le\lfloor(u+a)/2\rfloor$ and therefore
\[\sum_{i=1}^s(e_i-1)=a-s\ge a-\lfloor(u+a)/2\rfloor,\]
from which the result follows immediately.
\end{proof}

\subsection{The case of odd $n$}\label{odd_bounds_section}

Assume that $n$ is odd. Using Lemma \ref{genus_contrib_lower_bound_lem}, we now explain how to obtain lower bounds for the genera of subextensions of $N/K$. Define subsets $\Theta_n$ and $\Lambda_n$ of $\WW$ by
\begin{align*}\Theta_n&=\{\rho_i^s\;\vert\;1\le i\le r,\;0<s<n\},\\
\Lambda_n&=\{\rho_i^{-s}\tau_{i,j}\rho_i^s\;\vert\;1\le i<j\le r, \;0\le s<n\}.
\end{align*}

For every subgroup $H$ of $\WW$ and every divisor $d$ of $n$, let
\[u_{n,d}(H)=
\begin{cases}
(r-1)!n^r\#(H\cap\Theta_{n,d})/|H| & \;\text{if } d<n,\\
2(r-2)!n^{r-1}\#(H\cap\Lambda_n)/|H| & \;\text{if } d=n,
\end{cases}\]

and
\[g_{n,d}'(H)=(\deg\Delta_{n,d})\left\lceil |\WW:H|-\left\lfloor\frac{u_{n,d}(H)+|\WW:H|}{2}\right\rfloor\right\rceil.\]

Here, $\Theta_{n,d}$ denotes the set of elements of $\Theta_n$ having cycle type $(n/d,d)$.

\begin{prop}\label{genus_lower_bound_prop}
With notation as above, let $L$ be the fixed field of $H$. Then the genus of $L$ satisfies
\begin{equation}\label{genus_lower_bound_eq}
g(L)\ge\left\lceil1-|\WW:H|+\frac{1}{2}\sum_{d|n}\max(g_{n,d}'(H),0)\right\rceil.
\end{equation}
\end{prop}
\begin{proof}
Let $d$ be a proper divisor of $n$, and let $b\in R_{n,d}$. If $\P$ is a place of $N$ lying over $p_b$, Proposition \ref{dynatomic_inertia_prop} implies that the inertia group $I_{\P|p_b}$ is generated by an element $\gamma$ with cycle type $(n/d,d)$. By part (3) of Proposition \ref{conjugacy_centralizer_megaprop}, we have $\gamma=\rho_i^s$ with $1\le i\le r$ and $0<s<n$. Moreover, the order of the centralizer of $\gamma$ is given by $|C_{\WW}(\gamma)|=(r-1)!n^r$. Thus, by Corollary \ref{function_field_inertia_cor}, the number of places $\q\in\PP_L(p_b)$ such that $e_{\q|p_b}=1$ is equal to
\[(r-1)!n^rs(H,\gamma)/|H|,\]
where $s(H,\gamma)$ is the number of conjugates of $\gamma$ which belong to $H$. Note that every conjugate of $\gamma$ belongs to $\Theta_{n,d}$, so that $s(H,\gamma)\le\#(H\cap\Theta_{n,d})$. It follows that the number of places $\q\in\PP_L(p_b)$ such that $e_{\q|p_b}=1$ is bounded above by $u_{n,d}(H)$. Letting $a=|\WW:H|$, Lemma \ref{genus_contrib_lower_bound_lem} implies that
\[\sum_{\q\in\PP_L(p_b)}(e_{\q|p_b}-1)\ge\lceil a-\lfloor(u_{n,d}(H)+a)/2\rfloor\rceil.\]
Recalling the number $g_{n,d}(H)$ defined in \S\ref{computation_section}, the above inequality implies that
$g_{n,d}(H)\ge g_{n,d}'(H)$ and therefore $g_{n,d}(H)\ge\max(g_{n,d}'(H),0)$. 

By a similar argument we can show that  $g_{n,n}(H)\ge \max(g_{n,n}'(H),0)$. Let $b\in R_{n,n}$ and $\P\in\PP_N(p_b)$. Then $I_{\P|p_b}=\langle\gamma\rangle$, where $\gamma$ has cycle type $(2,n)$. Since $n$ is odd, part 2(b) of Proposition \ref{conjugacy_centralizer_megaprop} implies that $\gamma=\rho_i^{-s}\tau_{i,j}\rho_i^s$ for some $1\le i<j\le r$ and $0\le s<n$. Moreover, $|C_{\WW}(\gamma)|=2(r-2)!n^{r-1}$. The number of places $\q\in\PP_L(p_b)$ such that $e_{\q|p_b}=1$ is therefore given by
\[2(r-2)!n^{r-1}s(H,\gamma)/|H|.\]
Now, every conjugate of $\gamma$ belongs to $\Lambda_n$, so $s(H,\gamma)\le\#(H\cap\Lambda_n)$. The number of places $\q\in\PP_L(p_b)$ with $e_{\q|p_b}=1$ is thus bounded above by $u_{n,n}$. Letting $a=|\WW:H|$, we have
\[\sum_{\q\in\PP_L(p_b)}(e_{\q|p_b}-1)\ge\lceil a-\lfloor(u_{n,n}(H)+a)/2\rfloor\rceil,\]
which implies that $g_{n,n}(H)\ge g_{n,n}'(H)$. We have thus proved:
\[g_{n,d}(H)\ge\max(g_{n,d}'(H),0)\;\text{ for every divisor $d$ of $n$}.\]
Now \eqref{genus_lower_bound_eq} follows from the genus formula \eqref{genus_formula}.
\end{proof}

\begin{rem}\label{infty_bound_rem}
Note that in proving the bound \eqref{genus_lower_bound_eq} we have disregarded the contribution to the genus coming from ramified places lying over $p_{\infty}$. Though the bound would certainly be improved if these places were considered, doing so would substantially increase the amount of time and memory required to compute the bound. In particular, it would require determining the intersection $H\cap C$, where $C$ is the set of all conjugates in $\WW$ of the permutation $(1,2)(3,4)\cdots(r-1,r)$. Now, part 1(a) of Proposition \ref{conjugacy_centralizer_megaprop} implies that $\#C=\frac{n^{r}r!}{(r/2)!(2n)^{r/2}}\ge n^{r/2}$, which suggests that $C$ might be difficult to construct in practice. And indeed, our attempts to compute all the elements of $C$ in the case $n=7$ failed due to excessive memory requirements.
\end{rem}

\begin{rem}\label{intersection_size_rem}
In order to compute the number on the right-hand side of \eqref{genus_lower_bound_eq}, the key step is to determine the cardinalities of the sets $H\cap\Theta_{n,d}$ and $H\cap\Lambda_n$, which would be difficult to do if all the sets involved were quite large. Fortunately, while the group $H$ may be extremely large (for instance, $H$ might be the largest maximal subgroup of the Galois group of $\Phi_9$, in  which case $|H|\approx 9.73\times10^{127}$), the sets $\Lambda_n$ and $\Theta_{n,d}$ are small. Indeed, $\#\Theta_{n,d}\le\#\Theta_n=r(n-1)$ and $\#\Lambda_n=n\cdot{r\choose 2}$. This makes it computationally feasible to construct the sets $H\cap\Lambda_n$ and $H\cap\Theta_{n,d}$, and hence to compute the desired lower bound.
\end{rem}

We can now complete the proof of Theorem \ref{finiteness_theorem_intro}. The finiteness of $E_7$ and $E_9$ is proved by a series of computations carried out using \textsc{Magma}; the code used for these computations is available in \cite{krumm_finiteness_code}.

\begin{thm}
The sets $E_7$ and $E_9$ are finite.
\end{thm}
\begin{proof}
We consider first the case of $E_7$. The polynomial $\Phi_7$ has $D=126$ roots which can be partitioned into $r=18$ cycles. Thus, $\WW=(\Z/7\Z)\wr S_{18}$. Constructing the group $\WW$ and computing representatives for the conjugacy classes of maximal subgroups of $\WW$, we obtain 16 groups which we denote by $M_1,\ldots, M_{16}$. The sets $\Theta_7$ and $\Lambda_7$ are easily constructed; we find that $\#\Theta_7=108$ and $\#\Lambda_7=1071$. 

Let $L_i$ denote the fixed field of $M_i$. For each subgroup $M_i$ we compute the numbers $u_{7,7}(M_i)$ and $u_{7,1}(M_i)$, and use these to calculate $g_{7,7}'(M_i)$ and $g_{7,1}'(M_i)$. This is a trivial computation given the small size of the sets $\Theta_7$ and $\Lambda_7$. The inequality \eqref{genus_lower_bound_eq} then yields a lower bound for $g(L_i)$.

Carrying out these calculations, the lowest lower bound we obtain for the genera $g(L_i)$ is 6; hence $g(L_i)>1$ for every $i$, which implies that $E_7$ is finite. The total time required for all of the above computations is 0.42 s.

The proof of finiteness of $E_9$ follows the same steps as above. In this case the lowest lower bound we obtain for $g(L_i)$ is 4. Total computation time is 197 s, with 179 s spent computing the maximal subgroups of $\WW$.
\end{proof}

\subsection{The case of even $n$}\label{even_bounds_section}

In the case where $n$ is even, a bound similar to \eqref{genus_lower_bound_eq} can be proved; indeed, this only requires modifying the definition of the number $u_{n,n}(H)$. Unfortunately, when $n=8$ or 10 the bounds for the genera $g(L_i)$ obtained in this way are not greater than 1; in fact many of them are negative. We suspect, therefore, that most of the ramification in the extensions $L_i/K$ occurs over the place $p_{\infty}$. In order to improve the bounds for $g(L_i)$ we would have to determine the genus contribution coming from places lying over $p_{\infty}$. However, as discussed in Remark \ref{infty_bound_rem}, it is computationally infeasible to do this. Thus, we are unable to improve the bounds enough to show that $E_8$ and $E_{10}$ are finite.

\section{Density results}\label{density_section}

Having proved Theorem \ref{finiteness_theorem_intro}, we now turn our attention to Theorem \ref{density_thms_intro}. Recall that if $n$ is a positive integer and $c\in\Q$, we denote by $T_{n,c}$ the set of prime numbers $p$ such that the map $\phi_c(x)=x^2+c$ does not have a point of period $n$ in $\Q_p$. By applying Lemma \ref{galois_density_lem} below we will be able to calculate the density of $T_{n,c}$ for $n\in\{5,6,7,9\}$ and all but finitely many $c\in\Q$.

For every polynomial $F\in\Q[x]$, let $S_{F}$ be the set of all primes $p$ such that $F$ has a root in $\Q_p$. The Chebotarev density theorem implies that the density of $S_F$, which we denote by $\delta(S_F)$, exists and can be computed if the Galois group of $F$ is known. More precisely, we have the following result.

\begin{lem}\label{galois_density_lem}
Let $F\in \Q[x]$ be a separable polynomial of degree $D\ge 1$. Let $S$ be a splitting field for $F$, and set $G=\Gal(S/\Q)$. Let $\alpha_1,\ldots, \alpha_D$ be the roots of $F$ in $S$ and, for each index $i$, let $G_i$ denote the stabilizer of $\alpha_i$ under the action of $G$. Then the Dirichlet density of $S_F$ is given by
\begin{equation}\label{density_formula}
\delta(S_F)=\frac{\left|\bigcup_{i=1}^D G_i\right|}{|G|}.
\end{equation}
\end{lem}
\begin{proof}
This follows from Theorem 2.1 in \cite{krumm_lgp}.
\end{proof}

Note that for the purpose of computing $\delta(S_F)$ using the formula \eqref{density_formula}, the group $G$ may be replaced with any permutation group $\calG$ such that $G\equiv \calG$. Fixing a positive integer $n$, let $G_n$ be the Galois group of $\Phi_n$ over $\Q(t)$ and let $\calG=\Aut(\Gamma)$, where $\Gamma$ is the graph defined in \S\ref{graph_group_section}. Recall that $G_n\equiv \calG$.

\begin{lem}\label{counting_lem}
Let $\mathcal M$ be the set of all elements of $\calG$ having no fixed point. The cardinality of $\mathcal M$ is given by the formula
\[\#\mathcal M=\sum_{i=0}^r(n-1)^i\cdot n^{r-i}\cdot d(r,i),\]
where
\[d(r,i)={r\choose i}(r-i)!\sum_{k=0}^{r-i}\frac{(-1)^k}{k!}.\]
\end{lem}
\begin{proof}
The number $d(r,i)$ counts the permutations in $S_r$ which fix exactly $i$ elements of the set $\{1,\ldots, r\}$. The above formula for $d(r,i)$ is proved by an inclusion-exclusion argument; see Example 2.2.1 in \cite{stanley}.

For $0\le i\le r$, let $\mathcal M_i$ be the set of elements of $\mathcal M$ which fix exactly $i$ cycles of $\Gamma$. Clearly $\mathcal M$ is a disjoint union of the sets $\mathcal M_i$, so in order to prove the lemma it suffices to show that $\#\mathcal M_i=(n-1)^i\cdot n^{r-i}\cdot d(r,i)$.

Recall that every element $\sigma\in\calG$ has a unique representation of the form $\rho_1^{a_1}\cdots\rho_r^{a_r}\pi$, where $\pi\in S_r$ describes the action of $\sigma$ on the set of cycles of $\Gamma$, $\rho_k$ represents a $(1/n)$ rotation on the $k^{\text{th}}$ cycle, and $0\le a_k<n$.

Let $0\le i\le r$. Then an element $\sigma\in\calG$ represented as above belongs to $\mathcal M_i$ if and only if there exist indices $k_1,\ldots, k_i\in\{1,\ldots, r\}$ such that $\pi$ fixes $k_1,\ldots, k_i$ and has no other fixed points; and $a_{k_j}>0$ for $j=1,\ldots, i$. In constructing elements of $\mathcal M_i$ we therefore have $d(r,i)$ choices for $\pi$, $n-1$ choices for the exponents $a_{k_j}$, and $n$ choices for the remaining $r-i$ exponents. It follows that $\#\mathcal M_i=(n-1)^i\cdot n^{r-i}\cdot d(r,i)$, as required.
\end{proof}

\subsection*{Proof of Theorem \ref{density_thms_intro}}
Let $\Delta(t)$ be the discriminant of $\Phi_n$ and let
\[E=\{c\in\Q\;\vert\;\Delta(c)=0\}\cup E_5\cup E_6\cup E_7\cup E_9.\]
By the results of \S\S\ref{computation_section}-\ref{bounds_section}, $E$ is a finite set. Fix $n\in\{5,6,7,9\}$ and $c\in\Q\setminus E$. Since $c\notin E_n$, we have $G_{n,c}\cong G_n$. This implies that $G_{n,c}\equiv G_n$, where $G_{n,c}$ acts on the roots of $\Phi_{n}(c,x)$. Indeed, since $\Delta(c)\ne 0$, there is a subgroup $H$ of $G_n$ such that $G_{n,c}\equiv H$ (see Theorem 2.9 in \cite[Chap. VII]{lang_algebra}). By order considerations, $H$ must be equal to $G_n$.

Let $S_{n,c}$ be the set of primes $p$ such that $\Phi_n(c,x)$ has a root in $\Q_p$. The fact that $\Delta(c)\ne 0$ implies that $S_{n,c}$ is the complement of $T_{n,c}$. Indeed, every root of $\Phi_n(c,x)$ has period $n$ under $\phi_c$; see Theorem 2.4(c) in \cite{morton-patel}.

Since $G_{n,c}\equiv G_n\equiv\calG$, Lemma \ref{galois_density_lem} applied to $F(x)=\Phi_n(c,x)$ yields
\[\delta(S_{n,c})=\frac{\left|\bigcup_{\alpha\in\Gamma} \calG_{\alpha}\right|}{|\calG|},\]
where $\calG_{\alpha}$ is the stabilizer of $\alpha$ in $\calG$. It follows that $\delta(T_{n,c})=(\#\mathcal M)/|\calG|$, where 
$\mathcal M$ is defined as in Lemma \ref{counting_lem}. Using this lemma we obtain
\begin{align*}
\delta(T_{5,c})&=(9210721)/(6!5^6)\approx 0.8187,\\
\delta(T_{6,c})&=(3095578863701)/(9!6^9)\approx 0.8465,\\
\delta(T_{7,c})&\approx 0.8669,\\
\delta(T_{9,c})&\approx0.8948.
\end{align*}
This completes the proof of the theorem.
\qed

\section{The exceptional sets $E_n$}\label{exceptional_search_section}

We end this article with a brief discussion concerning the elements of the sets $E_n$. Recall the following notation introduced in \S\ref{preliminaries_section}: $S$ is a splitting field of $\Phi_n$ over $\Q(t)$; $G_n=\Gal(S/\Q(t))$; $M_1,\ldots, M_s$ are representatives of the conjugacy classes of maximal subgroups of $G_n$; and $\calX$ is the smooth projective curve with function field $S$. 

Our approach to proving the finiteness of $E_n$ for $n>4$ is based on Lemma \ref{explicit_HIT_lem}, which shows that $E_n$ is finite if every quotient curve $\calX/M_i$ has genus greater than 1. The proof of the lemma suggests that we may determine the elements of $E_n$ by finding a certain finite set $\mathcal E$ and determining all the rational points on the curves $\calX/M_i$. The set $\mathcal E$ as well as affine models for these curves can be obtained using the methods of the article \cite{krumm-sutherland}; however, the rational points on $\calX/M_i$ seem impossible to determine due to the large genera of the curves. (For instance, when $n=5$ one of the curves has genus 9526, as seen in the proof of Theorem \ref{E5_finiteness_thm}.) Hence, the problem of explicitly determining $E_n$ seems intractable at present. Nevertheless, it is possible to prove some basic results about the elements of $E_n$.
\begin{prop}
For every positive integer $n$ we have $\{0,-2\}\subseteq E_n$.
\end{prop}
\begin{proof}
For every $c\in\Q$, the polynomial $\Phi_n(c,x)$ divides $\phi_c^n(x)-x$, where $\phi_c(x)=x^2+c$. In particular, $\Phi_n(0,x)$ divides $x^{2^n}-x$, which implies that $\Phi_n(0,x)$ splits over a cyclotomic field. It follows that the Galois group $G_{n,0}$ is abelian, hence not isomorphic to $G_n$, since $G_n\cong (\Z/n\Z)\wr S_r$. Thus $0\in E_n$.

For $c=-2$ the polynomial $\phi_c$ is a Chebyshev polynomial satisfying
\[\phi_c(x+1/x)=x^2+1/x^2.\]

We claim that the polynomial $\Phi_n(-2,x)$ splits over the cyclotomic field $\Q(\zeta)$, where $\zeta$ is a primitive $(2^{2n}-1)^{\text{th}}$ root of unity; as above, this will imply that $-2\in E_n$. Suppose that $\alpha\in\bar\Q$ is a root of $\Phi_n(-2,x)$, and let $\beta\in\bar\Q$ satisfy $\beta+1/\beta=\alpha$. Then $\beta^{2^n}+1/\beta^{2^n}=\beta+1/\beta$, which implies that $(\beta^{2^{n}+1}-1)(\beta^{2^n-1}-1)=0$
and hence $\beta^{2^{2n}-1}=1$. Thus $\beta$, and therefore $\alpha$, belongs to $\Q(\zeta)$. This proves the claim.
\end{proof}

Given a positive integer $n$ for which $E_n$ is finite, one can attempt to find all the elements of $E_n$ by carrying out an exhaustive search within specified height bounds. Recall that the height of a rational number $a/b$ with $\gcd(a,b)=1$ is given by $\max(|a|,|b|)$. Fixing a height bound $h$, it a straightforward procedure to construct the set $B(h)$ of all rational numbers having height at most $h$. One can then construct all the polynomials $\Phi_n(c,x)$ for $c\in B(h)$, compute their Galois groups $G_{n,c}$ (for instance, using the algorithm of Fieker and Kl\"{u}ners \cite{fieker-kluners}, which is implemented in \textsc{Magma}), and check whether $G_{n,c}\cong (\Z/n\Z)\wr S_r$. The cost of carrying out this computation grows quickly with $n$, given the large degree of $\Phi_n$. For $n=7$ the degree of $\Phi_n$ is 126, and the above computation is very slow even for small height bounds $h$. However, for $n=5$ and 6 we have the following result.

\begin{prop}
Let $B(h)$ denote the set of all rational numbers with height at most $h$. Then 
\begin{align*}
E_5\cap B(50)&=\{-2, -16/9, -3/2, -4/3, -5/8, 0\},\\
E_6\cap B(20)&=\{-4, -2, 0\}.
\end{align*}
\end{prop}

\end{document}